\documentclass[preprint,12pt]{article}

\usepackage[top=30truemm, bottom=30truemm, left=25truemm, right=25truemm]{geometry}

\usepackage[dvips]{graphicx}
\usepackage{mathtools}

\usepackage{fancyhdr}
\usepackage{amsthm,color}
\usepackage{amsmath,amssymb,latexsym,amsfonts,mathrsfs}

\newcommand{\nn}{\nonumber}

\newcommand{\CAL}[1]{{\mathcal #1}}

\newcommand{\J}[1]{\left\langle #1 \right\rangle}

\newcommand{\pt}{\partial}     
             
\renewcommand{\th}{\theta}

\newcommand{\F}{\mathcal {F}}

\renewcommand{\S}{\mathcal {S}}

\newcommand{\re}{\mathbb R}
\newcommand{\R}{\mathbb R}

\newcommand{\N}{\nabla}

\newcommand{\al}{\alpha}

\newcommand{\gm}{\gamma}

\newcommand{\ep}{\varepsilon}
\newcommand{\lam}{\lambda}

\newcommand{\del}{\delta}
\newcommand{\Del}{\Delta}

\newcommand{\sg}{\sigma}
\renewcommand{\t}{\tau}

 \newtheorem{thm}{Theorem}[section]

 \newtheorem{lem}[thm]{Lemma}
 
 \newtheorem{prop}[thm]{Proposition}
 \newtheorem{rem}[thm]{Remark}
 
 \numberwithin{equation}{section}

\usepackage{amssymb}

\begin{document}

\begin{flushleft}
{ \Large \bf Asymptotic behavior of solutions to 
a dissipative nonlinear Schr\"odinger equation with time dependent harmonic potentials}
\end{flushleft}

\begin{flushleft}
{\large Masaki KAWAMOTO}\\
{Department of Engineering for Production, Graduate School of Science and Engineering, Ehime University, 3 Bunkyo-cho Matsuyama, Ehime, 790-8577. Japan }\\
Email: {kawamoto.masaki.zs@ehime-u.ac.jp}
\end{flushleft}
\begin{flushleft}
{\large Takuya SATO}\\
{Mathematical Institute, Tohoku University, Sendai, Miyagi, 980-8578, Japan \\
        E-mail: takuya.sato.b1@tohoku.ac.jp
}
\end{flushleft}
\begin{abstract}
We consider the Cauchy problem of a dissipative nonlinear Schr\"odinger equation
with a time dependent harmonic potential.
We find a critical situation that the $L^2$-norm of dissipative solutions decays or not and which is decided by 
a nonlinear power and time decay order of harmonic potential.\end{abstract}

\begin{flushleft}
\end{flushleft}

\section{Introduction}

We consider the Cauchy problem of the nonlinear Schr\"odinger equation
with a time dependent harmonic potential:
\begin{align}\label{eqn;dNLS}
\begin{cases}
& i \pt_t u - H_0(t)u = \lambda|u|^{p-1}u ,   
      t \in \re, x \in \re^n,\\
&  u(0,x)  =  u_0(x), 
      x \in \re^n,
     \end{cases}
\end{align}
where $\lam \in \mathbb{C} \setminus \{0\}$ with Im\,$\lam<0$,
$p>1$,
$u=u(t,x): \re \times \re^n \to \mathbb{C}$ is the unknown function 
and $u_0: \re^n \to \mathbb{C}$ is a given initial data.
Here, $H_0(t) = -\frac{1}{2}\Del + \frac{\sg(t)}{2} |x| ^2$,
$\sg(t)$ is a real valued function.


In case with $\sg \equiv 0$, i.e., the Cauchy problem of the nonlinear Schr\"odinger equation;
\begin{equation}\label{eqn;dNLS-K1}
\left\{
\begin{split}
& i \pt_t u + \frac12\Delta u = \lam|u|^{p-1}u ,   
     & t \in \re, x \in \re^n,\\
&  u(0,x)  =  u_0(x), 
     & x \in \re^n,
\end{split}
\right.
\end{equation}\noindent 
the local well-posedness of the problem 
\eqref{eqn;dNLS-K1}  was established 
by Ginibre-Velo \cite{G-V-JFA-1979},  
Kato \cite{K-AIHP-1987}, Y. Tsutsumi \cite{T-FE-1987} and Yajima \cite{Y-CMP-1987}
(see also \cite{C}, \cite{C-W-MM-1988}).
For  $\lam \in \re \setminus \{0\}$, 
it is well known that the exponent $p=1+2/n$ is a borderline
to indicate that  solutions scatter or not and this exponent is called as the Barab-Ozawa exponent (see \cite{B-JMP-1984}, \cite{O-CMP-1991}).
For $\lam \in \mathbb{C}$ with
Im\,$\lam < 0$,
the asymptotic behavior of the solution to \eqref{eqn;dNLS-K1}
have been studied by many authors recently.
In this  setting, 
the solution to \eqref{eqn;dNLS-K1} satisfies the dissipative energy identity:
\begin{align}
\label{eqn;mass}
   \|u(t)\|^2_{L^2}\, 
    + \,|{\rm Im}\,\lam|
          \int_0^t
             \|u(\tau)\|^{p+1}_{L^{p+1}}\,
          d\tau
     = \|u_0\|^2_{L^2},  \quad t \geq 0.
\end{align}
Hence, the $L^2$-norm of the corresponding solution is monotone decreasing 
in $t \geq 0$,
however it is whether the $L^2$-norm 
decays or not  as $t$ goes to infinity.
In recent works 
\cite{H-L-N-AMP-2016},
\cite{H-JMP-2019},
\cite{K-S-JDE-2007}, \cite{O-S-NoDEA-2020}, 
 \cite{S-AM-2020}, \cite{S-SN-2021}, \cite{S-CPDE-2006},
it is known that $p=1+2/n$ is the critical exponent to exhibit 
the $L^2$-decay of dissipative solutions to \eqref{eqn;dNLS-K1}.
The $L^2$-lower bound of the dissipative solution
is proved when $p > 1+ 2/n$ in
\cite{S-SN-2021}.
For $p \le 1+2/n$,
Kita-Shimomura \cite{K-S-JDE-2007} 
and Shimomura \cite{S-CPDE-2006}
observed the dissipative structure of \eqref{eqn;dNLS-K1} 
under Im\,$\lam<0$
and proved 
the $L^2$-decay  of dissipative solutions %
(cf. Sunagawa \cite{S-JMSJ-2006} for the case of the 
nonlinear Klein-Gordon equation).
Kita-Shimomura \cite{K-S-JMSJ-2009} removed
the restriction of the size of initial data
and Hayashi-Li-Naumkin \cite{H-L-N-AMP-2016}
showed the decay rate of the dissipative solution with large data
by imposing the strong dissipative condition
\begin{align}
\label{eqn;SD}
  \frac{p-1}{2\sqrt{p}}|\text{Re}\,\lam| 
   \le |\text{Im}\,\lam|,
\end{align}
where $p>1$ is the exponent of the power type
nonlinearity in the problem \eqref{eqn;dNLS-K1}.
The condition \eqref{eqn;SD} was considered by Liskevich-Perelmuter \cite{L-P-PAMS-1995}. Okazawa-Yokota \cite{O-Y-JDE-2002} applied \eqref{eqn;SD} 
to the theory of
the complex Ginzburg-Landau equation, 
where they showed the global existence
and smoothing property of the solution. Hayashi-Li-Naumkin \cite{H-L-N-AMP-2016} obtained  the $L^2$-decay rate 
of the dissipative solution under the condition \eqref{eqn;SD}.
Hoshino \cite{H-JMP-2019} showed the $L^2$-decay order of  
dissipative solutions  in the Sobolev space with a low differential order.
Ogawa-sato \cite{O-S-NoDEA-2020} and Sato \cite{S-AM-2020} showed the $L^2$-decay order of solutions which has analytic or Gevrey regularity in spatial variable.

Under the existence of $\sigma (t)$, Carles \cite{Ca} and Carles-Silva \cite{CS} considered the existence of global solution and blow-up issues for \eqref{eqn;dNLS} with more generalized time-dependent harmonic potentials such that $H_0(t) = -\Delta /2 + V(t,x)$ with condition $|V(t,x)| \leq C (1+ |x|^2 )$ as $|x| \gg 1$. In the paper \cite{Ca}, the time-in-local Strichartz estimates for \eqref{eqn;dNLS} with $\lambda =0$ also have been shown and which are applied in some proofs. On the other hand, Kawamoto-Yoneyama \cite{KY-JEE-2018} and Kawamoto \cite{Kawa} focussed on the case where the harmonic potential can be written as $\sigma (t) |x|^2$ with explicit expression of $\sigma (t) = c_1 t^{-2}$, $c_1 \in [0,1/4)$ for $|t| \gg 1$. In that case, the linear solution of \eqref{eqn;dNLS} has weak dispersive estimates compared with the linear solution of \eqref{eqn;dNLS-K1} and that possible to show global-in-time Strichartz estimates. By using such Strichartz estimates, the asymptotic behavior of solutions have been concerned 
 in Kawamoto \cite{K-JDE-2021}. In \cite{K-JDE-2021}, the critical case $c_1 = 1/4$ is also investigated and found that the threshold of scattering of solutions change by the situation of a time decaying harmonic oscillator (see also \cite{K-JMAA-2021}, \cite{KM-JEE-2021}). 

Our interest in this paper is to find the $L^2$-decay (mass decay) criterion for the problem \eqref{eqn;dNLS} with a dissipative nonlinear setting and a time dependent harmonic potential. 
If $\sg $ is positive constant, 
Antonelli-Carles-Sparber \cite{ACS-IMRN-2013} showed that the mass of dissipative solutions decays as $t \to \infty$ 
whenever $p>1$, namely the critical situation
for the mass decay is no longer appeared in the dissipative nonlinear Schr\"odinger equation.
In this paper, we consider the problem \eqref{eqn;dNLS} which contains  the time dependent harmonic potential and we reveal 
a relation between the nonlinear power, an effect of the time decaying potential and the mass decay of dissipative solutions.


%
%
Let $\S(\re^n)$ be the set of smooth and rapidly decreasing functions.
For any $f \in \S(\re^n)$, 
we define the Fourier transform
and the Fourier inverse transform, respectively.
\begin{align*}
   \F [f](\xi) 
    = \widehat{f}(\xi)
    \equiv \frac{1}{(2\pi)^{\frac{n}{2}}} 
                     \int_{\re^n}  e^{-ix \cdot \xi}f(x) \,dx, \quad
    \F^{-1} [f](x) 
       \equiv \frac{1}{(2\pi)^{\frac{n}{2}}} 
                     \int_{\re^n}  e^{ix \cdot \xi}f(\xi) \,d\xi.
\end{align*}

%
For $s>0$ and $r>0$,  let
\begin{align*}
  H^{s,r}(\re^n) \equiv
      \Big\{f \in H^s(\re^n);\,
           \|f\|_{H^{s,r}} \equiv   \|\J{x}^r f\|_{L^2}+\|\J{\N}^s f\|_{L^2}
       <\infty
    \Big\},
\end{align*}
where $\J{x} \equiv (1+|x|^2)^{1/2}$ and
$\J{\N}^s f \equiv \F^{-1}[\J{\xi}^s \widehat{f}\,]$
and we denote 
$H^{s,0}(\re^n)=H^s(\re^n)$.

Let us introduce the setting of a time dependent harmonic potential in the linear part:
\begin{align}
\label{k1}
  H_0(t) = -\frac{1}{2}\Del + \frac{\sigma (t)}{2} |x| ^2.
\end{align}
We define $U_0(t,s)$ as a propagator for $H_0(t)$, that is, a family of unitary operators on $L^2(\re^n)$ satisfying 
\begin{align*}
& i \partial _t U_0(t,s) = H_0(t) U_0(t,s) , \quad i \partial _s U_0(t,s) = - U_0(t,s) H_0(s), \\ 
& U_0(t,\tau) U_0(\tau ,s) = U_0(t,s), \quad U(s,s) = \mathrm{Id}_{L^2} 
\end{align*}
on $\mathscr{D} (-\Delta + |x|^2 )$.
Let $y_j(t)$, $j=1,2$ be a solution to
\begin{align}\label{eqn;sg-equation}
y''(t) + \sigma (t) y(t) = 0, \quad 
\begin{cases}
y_1(0) = 1, \\ 
y_1'(0) = 0 ,
\end{cases}
\quad 
\begin{cases}
y_2(0) = 0 , \\ 
y_2'(0) = 1. 
\end{cases}
\end{align}
We assume that $y_j(t)$ satisfies following conditions;
\begin{enumerate}
\vspace{1mm}
\item[(A).] There exist $c_0 >0$ and $T_0 \geq 0$ such that for any $t \geq T_0$,
it follows that $|y_2 (t)| \geq c_0$. \\ 
\vspace{-2mm}
\item[(B).] For all $t \geq T_0$, $y_j(t)$ and $y_j '(t)$, $j=1,2$ are continuous. \\ 
\vspace{-2mm}
\item[(C).] 
There exist $C>0$ and $\del>0$ such that for all $t \ge T_0$,
an asymptotic behavior is given by
\begin{align} \label{K10}
 \left| \frac{y_1(t)}{y_2(t)} \right| \le Ct^{-\del}.
\end{align}
\end{enumerate}

Here, we do not assume the some asymptotic conditions to $\sigma (t)$ but to $y_j (t)$, $j=1,2$. The reason why is the dispersive estimates for linear equations is characterized through $y_j (t)$, $j=1,2$, which was found by Korotyaev; 
\begin{lem} [Korotyaev \cite{Ko}]
For $\phi \in L^{1}({\R }^n)$, the following pointwise estimate holds: For any $t,s \in \re$,
\begin{align*}
\left\| 
U_0(t,s) \phi
\right\|_{L^{\infty}} \leq C \left| 
y_1(t) y_2(s) - y_1(s) y_2(t)  
\right|^{-\frac{n}{2}} \left\| \phi \right\|_{L^{1}}
\end{align*}
Particularly,  it holds in case with $s=0$ that
\begin{align*}
\left\| 
U_0(t,0) \phi
\right\|_{L^{\infty}} \leq C \left| 
 y_2(t)  
\right|^{-\frac{n}{2}} \left\| \phi \right\|_{L^{1}}.
\end{align*}
\end{lem} 

Physically, the classical position $x (t)$ of a quantum particle governed by $H_0(t)$ can be express as $x(t) = y_1(t) x_0 + y_2(t) p_0 /m $, where $x_0$, $p_0$ and $m$ are the initial position, the initial momentum and the mass of the quantum particle, respectively. Indeed, if $\sigma (t) \equiv 0$, i.e., the linear Schr\"{o}dinger operator, we have $y_1(t) = 1$ and $y_2(t) = t$, and get $x(t) = x_0 + t (p_0/m)$ is the uniform linear motion. Considering corresponding $\sigma (t)$'s, it appears some important physical models 
(see  \cite{IK-JMAA-2020}, \cite{K-JDE-2021}).

 We say $p > 1$ is the critical in the sense of $y_2(t)$, 
 if there exists $c_{+} >0$ such that
\begin{align}\label{eqn;critical}
\lim_{t \to \infty} t \left| y_2( t)\right|^{-\frac{n}{2}(p-1)} = c_{+}
\end{align}
and say the sub-critical in the sense of $y_2(t)$, if there exist $c_{+} >0$ and $\delta _{*} >0$ such that
\begin{align} 
\lim_{t \to \infty} t^{1- \delta _*}  |y_2(t) |^{-\frac{n}{2}(p-1) } 
 = c_{+}.
\label{eqn;sub-critical}
\end{align}
We also say the super-critical in the sense of $y_2(t)$, if there exist 
$c_{+} >0$ and $\delta^* >0$ such that
\begin{align} 
\lim_{t \to \infty} t^{1+\delta^*} \left| y_2(t) \right|^{-\frac{n}{2}(p-1) }  = c_{+}.
\label{eqn;super-critical}
\end{align}

In the case where $\sigma (t) \equiv 0$, we have $y_1(t) = 1$ and $y_2(t) =t$,
i.e., $p=1+2/n$ is the critical. 
In the case where $\sigma (t) = \sigma _0 t^{-2}$, $\sigma _0 \in [0, 1/4)$,
we have that for $t \geq T_0$,  
the functions $ t^{\lambda}$ and $ t^{1- \lambda}$ solve 
$y''(t) + \sigma (t)y(t) = 0$, 
where $\lambda = (1- \sqrt{1-4 \sigma _0})/2 \in [0,1/2)$. 
Choosing the suitable $\sigma (t)$ in $t\in [0,T_0)$, see, e.g. \cite{{K-JDE-2021}}, 
we get $y_1(t) = c_1 t^{\lambda}$ 
and $y_2(t) = c_2 t^{1- \lambda}$
where $c_1, c_2 \in \R \backslash \{ 0 \}$. 
In this case, it holds that $\lim t|y_2(t)|^{-\frac{n}{2}(p-1)}= |c_2|^{-\frac{n}{2}(p-1)} $, which yields $p=1 + 2/(n(1- \lambda))$ is the critical. In the appendix, we introduce some other models.
We state that the asymptotic behavior of dissipative solutions to \eqref{eqn;dNLS} 
vary by  the situation of $p$,
namely, the critical exponent for the $L^2$-decay of dissipative solutions 
depends on the decay order of the time decaying potential $\sg$.

 \begin{thm}[Small data case]\label{thm;L^2-decay}
Let $1 \le n \le 3$, $n/2<s<p$, 
$\lam \in \mathbb{C}$ with $\rm{Im}\,\lam < 0$,
and $\sg$ be the real function such that there exists fundamental solutions $y_1,y_2$
of the equation \eqref{eqn;sg-equation} 
satisfying $(A)$--$(C)$.
Assume that $p>1$ is the critical,
namely $p$ satisfies \eqref{eqn;critical}.
Then, there exists small $\del_0>0$ such that
for any $u_0 \in H^{s,s}(\re^n)$ with 
$\|u_0\|_{H^{s,s}} \le \del_0$,
the Cauchy problem \eqref{eqn;dNLS}
has a unique global solution $u \in C([0,\infty);H^{s}(\re^n))$
such that $U_0(t,0)^{} |x|^s U_0(t,0)^{-1}  u \in C([0,\infty);L^2(\re^n))$
and 
\begin{align}
\sup_{t \ge 0} \|u(t)\|_{H^{\frac{s}{2}}} < \infty.
\label{eqn;unif-s/2}
\end{align}
Moreover, the solution satisfies following decay estimates: For some $T_0$ and 
any  $t \ge T_0$,
\begin{align} 
    &\|u(t)\|_{L^\infty} \le \,\,
          C  |y_2(t)|^{-\frac{n}{2}}
             \Big(
                   \int_{T_0}^t |y_2 (\t)| ^{-\frac{n}{2}(p-1)}\,d\t
             \Big)^{-\frac{1}{p-1}},
          \label{eqn;Point-decay}\\
  &\|u(t)\|_{L^2} \le \,\,
         C \Big(\int_{T_0}^t | y_2(\t) |^{-\frac{n}{2}(p-1)}\,d\t
             \Big)^{-\frac{s}{(p-1)(s+n)}},
        \label{eqn;L^2-decay} 
\end{align}
On the other hand, assume that $p$ is the super-critical, 
namely $p$ satisfies \eqref{eqn;super-critical}. 
Then, for any non-trivial $u_0 \in H^{s,s}(\re^n)$ with $\|u_0\|_{H^{s,s}} \le \del_0$, 
the corresponding solution to \eqref{eqn;dNLS} has a uniformly lower and upper bounds such that for any $t \ge T_0$,
\begin{align}
\label{eqn;Lower-bound}
    e^{-C_{T_0}|{\rm Im}\,\lam|\|u_0\|^p_{H^{s,s}}}\|u_0\|_{L^2}
   \leq \|u(t)\|_{L^2}  \leq \| u_0 \|_{L^2} , 
\end{align}
where $C_{T_0}>0$ is independent of $t \geq T_0$.
\end{thm}

Even for the case $\sigma(t) \equiv 0$, the uniform estimate \eqref{eqn;unif-s/2} extends the result \cite{O-S-NoDEA-2020} to higher dimensions $n = 2,3$ and fractional derivatives $s \in (1,2)$. 
The key of such extension is uniform $H^s$ estimate \eqref{eqn;unif-s/2} which enable us to deal with the high-frequency parts for higher dimensions (see, section4).
Besides our results includes more generalized models of $\sigma (t)$ and if $\sg(t)$ appear, then the critical exponent for $L^2$-decay of solutions is different depending on the time decay  order of $\sg(t)$. 
Hayashi-Li-Naumkin \cite{H-L-N-AMP-2016} and Hoshino \cite{H-JMP-2019} showed an analogous uniform bound of solutions to \eqref{eqn;dNLS-K1} with $\sg \equiv 0$ by imposing special condition \eqref{eqn;SD}.
Theorem \ref{thm;L^2-decay} relax the dissipative restriction \eqref{eqn;SD}
and associated estimates \eqref{eqn;unif-s/2} and 
\eqref{eqn;L^2-decay} hold by only assuming  Im $\lam<0$ even if $\sg(t)$ is appearance.

The pointwise estimate \eqref{eqn;Point-decay} 
corresponds to the decay estimate of the dissipative problem \eqref{eqn;dNLS-K1} under the critical or sub-critical setting.
Hence, \eqref{eqn;Point-decay}  extends the results 
\cite{S-CPDE-2006}, \cite{K-S-JDE-2007}  and  \cite{K-S-JMSJ-2009} 
to the problem \eqref{eqn;dNLS}
which contains the time decaying harmonic potential.

\vspace{2mm}

If we assume further that $\lam \in \mathbb{C}\setminus\{0\}$
satisfies \eqref{eqn;SD},
then the solution to \eqref{eqn;dNLS} exists globally in time 
without any restriction of the size of initial data
and its $L^2$-norm decays for all dimensions $n \ge 1$. 
In the followings, we let $n \ge 1$ and $p>1$ for $n=1,2$ 
and $1<p < \frac{n}{n-2}$, $n \geq 3$, and define 
$\gamma = \gamma (n)$ so that $\gamma (1) = \frac12$ 
and $\gamma (n) = 1$ for $n \geq 2$.

\begin{thm}[Large data case]\label{thm;L^2-decay2}

Let $\lam \in \mathbb{C} \setminus \{  0 \}$ satisfies \eqref{eqn;SD}
and $\sg$ be the real function such that there exists fundamental solutions $y_1,y_2$
of the equation \eqref{eqn;sg-equation} 
satisfying $(A)$--$(C)$.
Assume that $p$ is the critical or sub-critical, namely $p$ satisfies \eqref{eqn;critical} or \eqref{eqn;sub-critical}, and that $\delta _* < \frac{\delta \theta}{2}$ 
with some $0< \theta < \frac{n}{2}(1-p) + \gamma p$.
Then, for any $u_0 \in H^{1,1}(\re^n)$, 
the Cauchy problem \eqref{eqn;dNLS}
has a unique global solution $u \in C([0,\infty);H^{1,1}(\re^n))$
such that $U_0(t,0) |x| U_0(t,0)^{-1} u \in C([0,\infty);L^2(\re^n))$
and it satisfies for any $t \ge 0$,
\begin{align}
\label{eqn;Upper-bound}
    \|\N u(t)\|_{L^2} \le  C\|\N u_0\|_{L^2}, \,\,\,
    \|U_0(t,0) |x| U_0(t,0)^{-1}  u(t)\|_{L^2} \le C\||x|u_0\|_{L^2},
\end{align}
where $C>0$ is independent of $t \ge 0$.
Moreover, there exists $C>0$ such that for any $t \ge T_0$, 
if $p$ is subcritical, 
\begin{align}
\label{eqn;L^2-decay2}
   \|u(t)\|_{L^2} 
      \le C \max\left\{ t^{- \frac{2 \delta _* }{(p-1)(2+n)}} , 
              t^{\delta _* - \frac{\delta \theta}{2}} \right\},
\end{align}
where $\del_*$ is given by \eqref{eqn;sub-critical} and if $p$ is critical, 
\begin{align*}
\|u(t)\|_{L^2}  \le C \left( \int_{T_0}^t |y_2(\t)|^{- \frac{n}{2} (p-1) } d \t \right)^{- \frac{2}{(p-1)(2+n) }}.
\end{align*}
\end{thm}

Now we consider the case where $\sigma (t) \equiv 0$. Then we see that $\delta = 1$ and $\delta _* = 1- \frac{n}{2} (p-1)$ since $y_1(t) = 1$ and $y_2(t) =t$. In that case, letting
\begin{align} \label{eqn;p(n)}
p(n)= 
\left\{
\begin{aligned}
&\frac{3+\sqrt{n^2+2n+9}}{n+2}, && n \ge 2,\\
&1+\sqrt{2},  &&n=1,
\end{aligned}
\right.
\end{align}
Hayashi-Li-Naumkin \cite{H-L-N-AMP-2016} found the decay estimate
such that for $p(n) < p < 1 + \frac{2}{n}$, 
\begin{align*}
 \|u(t)\|_{L^2}  \le C t^{- \frac{2 \delta _* }{(p-1)(2+n)}}, \quad  \left( \frac{2 \delta _* }{(p-1)(2+n)} = \left( \frac{1}{p-1} - \frac{n}{2} \right) \frac{2}{n+2} \right)
\end{align*}
and as far as we know, the case where $p < p(n)$ has not been known yet. 
Here in theorem \ref{thm;L^2-decay2} with $\sigma \equiv 0$, 
the necessary condition for $L^2$-decay is $\delta _* < \frac{\delta \theta}{2}$ and, together with  
$\delta = 1$ and $\delta _* = 1- \frac{n}{2} (p-1)$, 
that yields the lower bounds for $p$ such as $p>p_{**}(n)$, where $p_{**}(n)$ is given by
\begin{align}\label{eqn;p_{**}}
p_{**}(n) =
\begin{cases}
      2, & n=1, \\ 
      1, & n=2,\\
 \frac{4+n}{2+n}, & n \geq 3. 
\end{cases}
\end{align}
We notice that 
$p_{**}(n) < p(n)$ for $n=1,2,$ and also 
$\frac{4+n}{2+n} = \frac{3 + \sqrt{n^2+2n +1}}{2+n} <p(n)$
for $n \geq 3$. Moreover, setting
\begin{align} \label{eqn;p_*}
p_*(n) \equiv \frac{ n^2 + 3n + 6 + \sqrt{9n^2 + 28n + 36} }{(n+2)^2}, 
\end{align}
we find that $p_{**}(n) < p_*(n) < p(n)$. 
For $p_*(n) \leq p < p(n)$, we have 
$ t^{\delta _* - \frac{\delta \theta}{2}} \leq t^{- \frac{2 \delta _* }{(p-1)(2+n)}} $, which implies a natural extension of \cite{H-L-N-AMP-2016} in the sense of lower bounds for $p$;
\begin{thm}
Let $n \geq 1$, $\sigma (t) \equiv 0$ and that $p(n)$, $p_*(n)$ and $p_{**}(n)$ are equivalent to that of \eqref{eqn;p(n)}, \eqref{eqn;p_*} and \eqref{eqn;p_{**}}, respectively. 
Suppose that all conditions in Theorem \ref{thm;L^2-decay2} with $\sigma (t) \equiv 0$ are fulfilled. Then  the following decay estimate holds: 
\begin{align}\label{eqn;decay-0}
 \|u(t)\|_{L^2}  \le 
 \begin{cases}
 C t^{- \frac{2 \delta _* }{(p-1)(2+n)}}, & p_*(n) \leq p < 1 + \frac{2}{n}, \\ 
 C t^{\delta _* - \frac{\delta \theta}{2}}, & p_{**}(n)  < p < p_*(n).
 \end{cases}
\end{align}
\end{thm}

The decay estimate \eqref{eqn;decay-0}
contains the result \cite{H-L-N-AMP-2016} for $p>p(n)$ and also implies that
the dissipative solution to \eqref{eqn;dNLS-K1} shows the same $L^2$-decay
in \cite{H-L-N-AMP-2016} for $p_*(n) \le p \le p(n)$.
Even if $p_{**}(n) < p \le p_*(n)$, dissipative solutions show the $L^2$-decay
and $p_{**}(n)$ is appeared by estimating  approximate solutions to \eqref{eqn;dNLS-K1}
as $t \to \infty$. 
Namely, the estimate \eqref{eqn;decay-0} is improvement for the lower bound of $p$ to exhibit the $L^2$-decay of solutions to \eqref{eqn;dNLS-K1} under the sub-critical case.

\vspace{2mm}
Due to Theorem \ref{thm;L^2-decay}
and Theorem \ref{thm;L^2-decay2},
the integrability of $|y_2(t)|^{-n(p-1)/2}$ give a criterion for the asymptotic behavior of dissipative solutions to \eqref{eqn;dNLS}. 

%
%
The proof of Theorem relies on the reduction of \eqref{eqn;dNLS} to the following problem:
\begin{equation}\label{eqn;NLSv}
\left\{
\begin{split}
& i \pt_t v + \frac{1}{2y_1(t)^2}\Del v 
          = \lambda |y_1(t)|^{-\frac{n(p-1)}{2}}|v|^{p-1}v ,   
     & t >T_0,\, x \in \re^n,\\
&  v(T_0,x)  =  v_0, 
     & x \in \re^n,
\end{split}
\right.
\end{equation}
where $v_0=e^{-iy_1(T_0)y_1'(T_0)|x|^2/2}e^{i\log |y_1(T_0)|A}u_0(x)$, $A = (x \cdot (-i \nabla) + (-i \nabla) \cdot x)/2$ is the generator of the dilation group. 
In virtue of this transformation,
one enable to use the decomposition of the associated semigroup to observe the asymptotic behavior of dissipative solutions.      
Namely,  the solution to \eqref{eqn;NLSv} 
is well approximated by 
\begin{align}
\label{eqn;appro}
  \F [U_Y(t)^{-1}v](t,\xi), \quad U_Y(t) = e^{i \frac{y_2(t)}{2y_1(t)} \Delta}
\end{align}
for large $t$ and any $\xi \in \re^n$. Here $Y(t)=y_2(t)/2y_1(t)$
and the approximated solution \eqref{eqn;appro}
satisfies
\begin{align}
\label{eqn;pro-sol}
   \frac{1}{2}\pt_t 
        \Big|\F [U_Y(t)^{-1}v](\xi)\Big|^2
           =& -\frac{|\text{Im}\, \lam|}{|y_2(t)|^{\frac{n(p-1)}{2}}}
                          \Big|\F [U_Y(t)^{-1}v](\xi)\Big|^{p+1}
      \\ \nn           &\,\,+O(|y_2(t)|^{-\frac{n(p-1)}{2}}Y(t)^{-\frac{\th}{2}}),
       \quad  t \to \infty,
\end{align}
for $0<\th<1$
(cf. \cite{H-L-N-CPAA-2017}, \cite{K-L-S-DIE-2014}, \cite{K-AA-2016}). If a top term (first term) is dominant compared with a remainder term (second term)
in \eqref{eqn;pro-sol}, 
by solving \eqref{eqn;pro-sol}, we formally see that
\begin{align}
\label{eqn;approxi-sol}
  \Big|\F [U_Y(t) ^{-1}v](\xi)\Big|^2
       \simeq 
         \bigg(  
             \frac{|\F [U_Y(T_0)^{-1}v](\xi)|^{p-1}}
                 {1+|{\rm Im}\,\lam |
                 |\F [U_Y(T_0)^{-1}v](\xi)|^{p-1} 
               \int_{T_0}^t |y_2(\tau)|^{-\frac{n}{2}(p-1)}\,d\tau}
         \bigg)^{\frac{2}{p-1}}
\end{align}
as $t \to \infty$. If the remainder term is dominant, we does not have the expression \eqref{eqn;approxi-sol}.
However one can obtain the $L^2$-decay property of dissipative solutions by extracting a time decay from the remainder term.

If $|y_2(t)|^{-\frac{n}{2}(p-1)}$ 
in the denominator of the right side of \eqref{eqn;approxi-sol} 
is integrable,
the approximated solution \eqref{eqn;appro} 
does not decay and hence the solution itself,
while the solution decays if $ |y_2(t) |^{-\frac{n}{2}(p-1)}$ is not integrable 
with $t=\infty$.

\section{Preliminaries}
We start to compare the nonlinear powers $p(n)$ and $p_*(n)$ 
defined by \eqref{eqn;p(n)} and \eqref{eqn;p_*}.

\begin{lem}\label{lem;compare}
For $n \ge 1$, let $p(n)$ and $p_*(n)$ be defined by \eqref{eqn;p(n)} and \eqref{eqn;p_*}, respectively.
Then $p(n)>p_*(n)$.
\end{lem}

\begin{proof}[Proof of Lemma \ref{lem;compare}]
In the case where $n=3,4$, they are calculated as
\begin{align*}
p(3) = \frac{3 + \sqrt{24}}{5} >  \frac{24 + \sqrt{201}}{25} = p_{\ast}(3) , \quad p(4) = \frac{3 + \sqrt{33}}{6} > \frac{34 + \sqrt{252}}{36} =p_{\ast} (4) .
\end{align*}
As for $n \geq 5$, by employing the Bernoulli inequality; 
$\sqrt{1+x} \le 1+\frac{1}{2}x$ for any $x \ge 0$, we compute
\begin{align*}
  (n+2)^2 (p(n) - p_{\ast}(n)) &= (n+2) \sqrt{n^2 +2n +9} - n^2 -3n \sqrt{1 + \frac{28}{9n} + \frac{4}{n^2}}
  \\ & \geq (n+2) \sqrt{n^2 +2n +9} - \left( n^2 + 3n + \frac{14}3 + \frac{6}{n} \right). 
\end{align*}
For $n \geq 5$, we have that $n^2 + 3n + \frac{14}3 + \frac{6}{n} \leq n^2 + 3n +6$ and which implies 
\begin{align*}
 (n+2)^2 (p(n) - p_{\ast}(n)) \geq \sqrt{(n^2 + 3n +6)^2 + 8n } - (n^2 + 3n +6) > 0.
\end{align*}
Thus we show $p_*(n) < p(n)$.
\end{proof}

Before considering the  $L^2$-decay, it is necessary to show the existence of global-in-time solutions to \eqref{eqn;dNLS} and hence we first divide the equation \eqref{eqn;dNLS} into two parts for the time interval; 
in the case where $t \in [0, T_1)$, $T_1 > 2T_0$, we obtain the time local solution to \eqref{eqn;dNLS} 
via the approach due to Kawamoto-Muramatsu \cite{KM-JEE-2021}.
\begin{prop}[\cite{KM-JEE-2021}]\label{p1}
Let $ n/2< s <p$. Then for any $T_1$,
there exists $\ep  _{0,0} = \ep_{0,0} (T_1) > 0$ such that for any $u_0$ satisfying 
$\| u_0 \|_{H^{s,s}} = \ep_{0,0} $, 
there uniquely exists a solution to \eqref{eqn;dNLS} and $0< \ep_0 \ll 1$ such that 
$u \in C([0,T_1]; H^{s,s})$ and 
\begin{align*}
\sup_{t \in [0,T_1]}\left\| u(t, \cdot ) \right\|_{H^{s,s}} 
   \le \ep_0 \| u_0 \|_{H^{s,s}}.
\end{align*} 
\end{prop}
On the other hand, the case where $t \in [T_1, \infty)$, to consider the mass decay property and the existence of the global-in-time solution at the same time, we reduce the equation by using unitary transform. Thanks to this reduction, one enable to use the approaches in some previous works. By the sub-consequence of the decomposition scheme by Korotyaev \cite{Ko}, we have the following;

\begin{lem}\label{lem;u-v}
Let $T_0>0$, $u(t,x) $ and $y_0(t)$ be a solution to \eqref{eqn;dNLS} and a solution to $y''(t) + \sigma (t) y(t) =0$, respectively. Assume that there exist $c>0$ and $T_0 \geq 0$ such that $|y_0(t)| \geq c$ and $y_0(t)$, $y_0 '(t)$ are continuous for all $t \geq T_0$. Define $v(t,x)$ as 
\begin{align*}
v(t,x) = e^{-i y_0 (t) y_0'(t) |x|^2 /2 } e^{i ( \log |y_0(t)| A )} u(t,x) , 
\qquad A = -\frac{i}{2} \left( x \cdot \nabla + \nabla \cdot x \right)
\end{align*} 
with the unitary operators $e^{-i y_0 (t) y_0'(t) |x|^2 /2 }$
and $ e^{i ( \log |y_0(t)| A )}$ on $L^2(\R ^n)$. Then $v(t,x) $ satisfies 
\eqref{eqn;NLSv} with 
$v_0=e^{-iy_0(T_0)y_0'(T_0)|x|^2/2}e^{i\log |y_0(T_0)|A}u_0(x)$.
Moreover, the mass conserves  
\begin{align} \label{1}
  \|u(t)\|_{L^2}=\|v(t)\|_{L^2}, \quad t \ge T_0
\end{align}
and it follows the $L^\infty$-identity:
\begin{align}\label{2}
  \left\| u(t) \right\|_{L^\infty} 
       = |y_0(t)|^{-\frac{n}{2}}  \left\| v(t) \right\|_{L^\infty},
  \quad t \ge T_0.
\end{align} 
\end{lem}

\begin{proof}
Let $y_0 (t)$ be a solution to $y''(t) + \sigma (t) y(t) =0 $ and $D_x = -i \nabla $. Straightforward calculation shows 
\begin{align*}
i \frac{\pt}{\pt t} v(t,x) = A_1(t) + A_2(t)
\end{align*}
with 
\begin{align*}
A_1 (t) \equiv i \frac{\partial }{\partial t} 
                 \left( e^{-i y_0(t) y_0'(t) |x|^2/2} e^{i \log |y_0 (t)| {A}} \right) u(t,x) 
             + e^{-i y_0(t) y_0'(t) |x|^2/2} e^{i \log |y_0 (t)| {A}} H_0(t) u(t,x)
\end{align*}
and 
\begin{align*}
A_2 (t) = \mu e^{-i y_0(t) y_0'(t) |x|^2/2} e^{i \log |y_0 (t)| {A}}  
               | u(t,x) |^{\theta} u(t,x).
\end{align*}
At first, we calculate $A_1 (t)$. Let ${J}_0 (t) := \left( e^{-i y_0(t) y_0'(t) |x|^2/2} e^{i \log |y_0 (t)| {A}} \right)$. On ${S}(\re^n)$, we have
\begin{align*}
 i \frac{\partial }{\partial t} {J}_0 (t) &= \frac{( (y_0')^2 + y _0y _0 '') |x|^2 }{2} {J}(t) -  \frac{y_0 '}{y_0} \left( e^{-i y_0(t) y_0'(t) |x|^2/2} {A} e^{i \log |y_0 (t)| {A}} \right)
 \\ &= \left( \frac{( (y_0')^2 + y _0y _0 '') |x|^2 }{2}
             -\frac{y_0' (x \cdot (D_x + y_0y_0'x) +  (D_x + y_0y_0'x) \cdot x )}
               {2y_0} \right) {J}_0 (t)
 \\ &= 
  \left( \frac{(  y _0y _0 ''- (y_0')^2 )  }{2} |x|^2 - \frac{y_0'}{y_0}  {A}  \right) {J}_0 (t)
\end{align*} 
On the other hand, commutator calculations on ${S}(\re^n)$ yields (e.g., see, Ishida-Kawamoto \cite{IK-RMP-2020} {4}) yields
\begin{align*}
e^{ i \beta {A}} |x|^2 e^{ -i \beta {A}} =  e^{2 \beta} |x| ^2 , \quad  e^{ i \beta {A}} |D_x|^2 e^{ - i \beta {A}}  = e^{-2 \beta} |D_x|^2 
\end{align*}
and by using this equation, we have
\begin{align*}
{J}_0(t) |D_x|^2 &= \frac{1}{y_0^2} e^{-i y_0(t) y_0'(t) |x|^2/2} |D_x|^2 e^{i \log |y_0 (t)| {A}} 
\\ &= \frac{1}{y_0^2} \left( D_x + y_0 y_0' x \right)^2 {J}_0(t)
\end{align*}
and 
\begin{align*}
{J}_0(t) |x|^2 = y_0 ^2 |x|^2 {J}(t).
\end{align*}
Therefore we have 
\begin{align*}
{J}_0(t) H_0(t) = \left( \frac{|D_x|^2}{2 y_0 ^2} + \frac{y_0 '}{y_0} {A} + \frac{(y_0')^2}{2} |x|^2 + \frac{y_0 ^2 \sigma }{2} |x|^2  \right) {J}_0(t)
\end{align*} 
and 
\begin{align*}
A_1(t) &= \left( \frac{|D_x|^2 }{2 y_0 (t)^2 } + \frac{y_0 (t)}{2}\left( y_0 ''(t) + \sigma (t) y_0(t) \right) |x|^2 \right) {J}_0(t) u(t,x) \\ 
&= -\frac{\Delta }{2 y_0 (t)^2 } v(t,x), 
\end{align*}
where we use the fact that $y_0$ solves \eqref{eqn;sg-equation}. As for $A_2(t)$, we employ the equation on $\phi \in {S}(\re^n) $ that
\begin{align*}
\left(e^{i \log |y_0 (t)| {A}} \phi \right) (x) 
    = |y_0(t)|^{n/2} \phi \left({y_0(t) x} \right),  
\end{align*} 
see e.g., Kawamoto-Yoneyama \cite{KY-JEE-2018} {2}. 
By this equation, it holds that
\begin{align*}
  {J}_0(t) | u(t,x) |^{p-1} u(t,x) 
      &= {|y_0 (t)|^{n/2} }e^{-i y_0(t) y_0'(t) |x|^2/2} 
             |u(t,y_0(t) x)| ^{p-1} u(t,y_0(t) x)
\\ &= e^{-i y_0(t) y_0'(t) |x|^2/2} \left| |y_0(t)|^{-n/2} 
              e^{i \log |y_0 (t)| {A}} u(t,x) \right| ^{p-1}  
           e^{i \log |y_0 (t)| {A}} u(t,x) 
\\ &= 
            |y_0 (t)|^{-n(p-1)/ 2} |v(t,x)|^{p-1} v(t,x).
\end{align*}
for all $t \ge T_0$ and $x \in \re^n$.
\end{proof}
\begin{rem}
The Lemma \ref{lem;u-v} is true for any $y_0$ satisfying the conditions in Lemma \ref{lem;u-v}. On the other hand, the $w$ solution to the linear part of reduced equation 
\begin{align*}
i \partial_t w + \frac{1}{2y_0 (t)^2}\Del w = 0
\end{align*}
can be written as 
\begin{align*}
w(t, \cdot) = e^{ i \frac{1}{2}\Delta \int _{T_0}^t \frac{1}{y_0 (s)^{2}} ds   } w(T_0, \cdot)
\end{align*}
and to characterize the mass decay/conserve properties with using $\int _{T_0}^t y_0 (s)^{-2} ds $ is complicated. Since the arbitrariness of the choice of $y_0(t)$, we replace $y_0$ as $y_1$. Then one finds 
\begin{align*}
\frac{d}{dt} \frac{y_2(t)}{y_1(t)} = \frac{1}{y_1(t)^2}
\end{align*}
and hence $w(t)$ with replacement $y_0$ to $y_1$ can be written as $w((t, \cdot) = e^{i \Delta ( Y(t) - Y(T_0) )} w(T_0, \cdot)$, where $Y(t) = (y_2(t)/ (2y_1(t)) )$. This expression of $w$ enable us to characterize the mass decay/conserve properties with only using the asymptotic behavior of $y_1(t)$ and $y_2(t)$ in $t$. 
\end{rem}

In the following, we set 
\begin{align*}
U(t,s)  = e^{i (Y(t) -Y(s) ) \Delta } =: U_Y(t) U_Y(s)^{-1}, \quad Y(t) \equiv \frac{y_2(t)}{2 y_1(t)}, \quad 
\CAL{J}(t) =  e^{-i y_1(t) y_1'(t) |x|^2/2} e^{i \log |y_1 (t)| {A}}.
\end{align*}
Then, we have the following lemma;
\begin{lem}
Let $T_0>0$ and 
$\tilde{v}(t,x) \equiv U(t,T_0) \tilde{v}_0$.
 Then $\tilde{v}=\tilde{v}(t,x)$ solves 
\begin{align}
\label{eqn;LSv}
\begin{cases}
   &i \partial _t \tilde{ v} +  \frac{1}{2y_1 (t)^2 } \Delta \tilde{ v}=0, 
        \quad t > T_0,\, x \in \re^n, \\
   &\tilde{v}(T_0,x) = \tilde{v}_0(x),
        \quad  x \in \re^n
\end{cases}
\end{align}
and an operator $U_Y(t)$ is decomposed as 
\begin{align}\label{eqn;MDFM}
U_Y(t) \tilde{v}_0=e^{ i \frac{y_2(t)}{2 y_1 (t)}\Del} \tilde{v}_0  
    = {\mathcal M}\left( \frac{y_2(t)}{2y_1(t)} \right) 
           \mathcal{D} \left(  \frac{y_2(t)}{2y_1(t)} \right) 
           \mathcal{F} 
     \mathcal{M} \left(  \frac{y_2(t)}{2y_1(t)} \right) \tilde{v}_0
\end{align}
 for all $t \geq T_0$.
\end{lem}
This lemma can be shown by combining the identity $ U(T_0,T_0) = \mathrm{Id}_{L^2}$, the uniqueness of the propagator $U(t, T_0)$ and the well-known decomposition formula that for any $s \in \R $, $e^{is \Delta /2} = \mathcal{M}(s) \mathcal{D}(s) \mathcal{F} \mathcal{M} (s)$.

Noting these, we consider the nonlinear equation associated the linear problem \eqref{eqn;LSv}.
\begin{equation}\label{eqn;NLS-y_1}
\left\{
\begin{split}
& i \pt_t v + \frac{1}{2y_1(t)^2}\Del v 
          = \lam |y_1(t)|^{-\frac{n(p-1)}{2}}|v|^{p-1}v ,   
     & t >T_0,\, x \in \re^n,\\
&  v(T_0,x)  =  v_0, 
     & x \in \re^n,
\end{split}
\right.
\end{equation}
where $y_1$ and $y_2$ are fundamental solutions to \eqref{eqn;sg-equation}.
Here we remark that $2 y_1(t) Y(t) = y_2(t) $.

Let $u$ be the solution to 
the problem \eqref{eqn;dNLS} on $[0, T_1] $, $T_1 > 2T_0$.
Take $T_0$ so that Assumption (A)--(C) hold and fix it. Then thanks to Proposition \ref{p1}, there exist a solution to \eqref{eqn;dNLS} satisfying $u (t, \cdot) \in C([0,T_1] ; H^{s,s} (\R ^n))$ and a sufficient small constant $\ep _0 > 0$ so that 
\begin{align*}
\sup_{t \in [T_0,T_1]}\left\|
u(t, \cdot)
\right\|_{H^{s,s}} \leq \ep _0.
\end{align*}
Here for all $\phi \in H^{s,s}$ and $t \in [T_0,T_1]$, 
\begin{align}
\nn 
\left\| 
\CAL{J}(t) \phi 
\right\|_{H^{s,s}} &\leq C \left\| (|-\Delta |^{s/2} + |x|^s +1 ) \CAL{J}(t) \phi \right\|_{L^2}\\ 
\nn &= 
C \left\| \left( |-\Delta |^{s/2} + |x|^s +1  \right) e^{-iy_1(t)y_1 '(t) |x|^2/2}e^{i \log |y_1(t)| A}  \phi \right\|_{L^2} 
\\ 
\nn &= 
C \left\| 
\left(  \left| -i \nabla y_1(t)- y_1'(t) x \right|^s + |y_1(t)|^{-s} |x|^s + 1 \right) \phi
\right\|_{L^2}
\\ 
\label{5} & \leq \hat{C}_0 ({T_1}) \left\| \phi \right\|_{H^{s,s}}
\end{align} 
holds, where we use $| y_1 (t) | \geq c_0$ for $t \geq T_0$ and interpolation between 
\begin{align*}
\left\| | -i a \nabla + b x |^0 \cdot {\bf 1} \cdot \J{-\Delta + |x|^2 +1}^{0} \right\| \leq \tilde{C}_0(T_1)
\end{align*}
and 
\begin{align*}
\left\| | -i a\nabla +  b x |^2 \cdot {\bf 1} \cdot \J{-\Delta + |x|^2 +1}^{-1} \right\| \leq \tilde{C}_1(T_1), 
\end{align*}
see, e.g., Kato \cite{Kato}, where $a,b \in \R$. With regard to the above inequality, \eqref{1} and \eqref{2}, we also have the local-in-time solution $v(t)$ on $t \in [T_0,T_1]$ to \eqref{eqn;NLSv} and that has the properties for $t\in [T_0,T_1]$;  
\begin{align}
\nn
\J{Y(t)}^{n/2}
\left\| 
v(t)
\right\|_{L^\infty} &= \J{Y(t)}^{n/2} \left\| \CAL{J}(t)^{} u(t) \right\|_{L^\infty} \\ 
\nn & = \J{Y(t)}^{n/2} |y_1(t)|^{n/2} \left\| u(t)\right\|_{L^\infty} \\ 
\label{6} & \leq \hat{C}_1(T_1) \left\| u(t)\right\|_{L^\infty},   
\end{align}
\begin{align}
\left\| v(t) \right\|_{H^{s,s}}  \leq \left\| \CAL{J}(t)^{} u(t) \right\|_{H^{s,s}} 
 \leq \hat{C}_0 (T_1) \left\| u(t) \right\|_{H^{s,s}} \label{7}
\end{align}
and 
\begin{align}
\label{8}
\left\|
|J|^s(t) v(t)
\right\|_{L^2} \leq C(T_1) \left\| \CAL{J}(t)^{} u(t) \right\|_{H^{s,s}} \leq \hat{C}_{2}(T_1) \left\| u(t) \right\|_{H^{s,s}}.
\end{align}
Rough calculation demands the $T_1 $ dependence for constants $\hat{C}_j (T_1)$, $j = 0,1,2$ and hence in the followings we remove the $T_1$ dependence of which by using the energy estimate for the associated problem \eqref{eqn;dNLS}.

We define an operator which provide a dispersive time decay of solutions to \eqref{eqn;dNLS} by 
\begin{equation}\label{eqn;J}
J(t)f =U_Y(t) x U_Y\left( t\right)^{-1} f
= \CAL{M}(Y(t)) iY(t) \nabla \CAL{M}^{-1} (Y(t))f = (x + iY(t)\nabla)f, 
\end{equation}
and the fractional power of $J(t)$  is defined that for $0<\gm<1$,
\begin{equation}\label{eqn;J-frac}
   |J|^\gm(t)f = U_Y\left( t\right) |x|^\gm U_Y\left( t\right)^{-1}f 
      = \CAL{M}(Y(t)) |Y(t)|^\gm |D_x|^\gm M^{-1}f,  
\end{equation}
where $|D_x|^\gm f= \mathcal{F}^{-1}[ |\xi|^\gm \widehat{f}]$. 
Note that the operators $J(t)$ and $|J|^\gm(t)$ commute with 
$i\partial_t + \frac{1}{2 y_1(t)^2} \pt^2_x$. 
Then, the unitary operator $U_Y(t)$ has the following decay property
(cf. Ozawa \cite{O-CMP-1991}, Hayashi-Naumkin \cite{HN-AJM-1998}, 
Kita-Shimomura \cite{K-S-JDE-2007}, 
\cite{K-S-JMSJ-2009}, see also \cite{H-JMP-2019},  \cite{K-AA-2016}, \cite{K-L-S-DIE-2014}, \cite{S-CPDE-2006}).

\begin{lem}
\label{lem;R-decay}
Let $n \geq 1$, $s>n/2$, $s_1 = \min \{ s-n/2, 1 \} $ and $p$ be the super-critical given by \eqref{eqn;super-critical}. 
Then, there exists $C>0$ such that for any  
$f \in \S(\re)$ and $t \geq T_0$,
the following pointwise estimate holds:
\begin{align}
\begin{aligned}
  \Big\|
    \F [U_Y(t)^{-1}&(|f|^{p-1}f)]
      - |y_2(t)| ^{-\frac{n}{2}(p-1)}|\F [U_Y(t)^{-1})f]|^{p-1}(\F [U_Y(t)^{-1}f])
  \Big\|_{L^\infty}\\
 \le& \,C|y_2(t)|^{-\frac{n}{2}(p-1)} |Y(t) |^{-s_1}
     (\|f\|_{L^2}+\||J|^s(t)f\|_{L^2})^p.  \label{eqn;R-decay1}
\end{aligned}
\end{align}
Moreover, the following $L^2$-decay property  holds: 
Suppose $1<p < \infty$, if $n=1,2,$ and 
$ 1< p< \frac{n}{n-2}$ if $n \geq 3$. Then we have
\begin{align}
\begin{aligned}
  \Big\|
    \F [U_Y(t)^{-1}&(|f|^{p-1}f)]
      - |y_2(t) |^{-\frac{n}{2}(p-1)}|\F [U_Y(t)^{-1}f]|^{p-1}
          (\F [U_Y(t)^{-1} f])
  \Big\|_{L^2}\\
       \le& \,C |y_2(t)| ^{-\frac{n}{2}(p-1)}|Y(t)|^{-\frac{\th}{2}}
            (\|f\|_{L^2}+\|J(t)f\|_{L^2})^p,  
     \label{eqn;R-decay2}
\end{aligned}
\end{align}
where 
$\th =1$ if $n=1$, $0<\th <1$ if $n=2$, 
and $\th=\frac{n}{2}-\frac{n-2}{2}p$ if $n \ge 3$.
Here $J(t)=x+iY(t)\N$ and $Y(t)=y_2(t)/2y_1(t)$.
\end{lem}
\begin{proof}[Proof of Lemma \ref{lem;R-decay}]
The proof of \eqref{eqn;R-decay1} and \eqref{eqn;R-decay2} relies on 
the previous works \cite{HN-AJM-1998},  \cite{K-L-S-DIE-2014},  \cite{K-S-JDE-2007}, \cite{K-S-JMSJ-2009}, \cite{H-JMP-2019}. We only prove \eqref{eqn;R-decay2} briefly
(see \cite{H-JMP-2019} for details).
Let the left side in \eqref{eqn;R-decay2} be $R(t,\xi)$.
Then, $R$ is decomposed by  $R = R_1+R_2$, where
\begin{align*}
  & R_1(t,\xi)=\,\frac{\lam}{|y_2(t)|^{\frac{n}{2}(p-1)}}
                     \F (e^{-i\frac{|x|^2}{2Y(t)}}-1)\F^{-1}
                       \Big(|\F [e^{-i\frac{|x|^2}{2Y(t)}}U_Y^{-1}f]|^{p-1}
                            (\F [e^{-i\frac{|x|^2}{2Y(t)}}U_Y^{-1}]f)\Big),\\
  & R_2(t,\xi)=\,\frac{\lam}{|y_2(t)|^{\frac{n}{2}(p-1)}}
              \Big(
                    |\F e^{-i\frac{|x|^2}{2Y(t)}}U_Y^{-1}f|^{p-1}
                    (\F e^{-i\frac{|x|^2}{2Y(t)}}U_Y^{-1}f)
                   - |\F U_Y^{-1}f|^{p-1}(\F U_Y^{-1}f)
             \Big).
\end{align*}

Let $n \ge 3$. By the inequality;
$\displaystyle
    |e^{-i|x|^2/2Y(t)}-1| \le C(|x|^2/|Y(t)|)^{\theta/2}
$ 
for $\theta =\frac{n}{2}-\frac{n-2}{2}p \in (0,1)$, we have
\begin{align*}
\|R_1(t)\|_{L^2} 
  \le&\,C|Y(t) |^{-\frac{\th}{2}}
        \Big\||\N|^\th
                \Big(|\F [e^{-i\frac{|x|^2}{2Y(t)}}U_Y^{-1}f]|^{p-1}
                   (\F [e^{-i\frac{|x|^2}{2Y(t)}}U_Y^{-1}]f)\Big)
         \Big\|_{L^2}\\
    \le&\, C |Y(t) |^{-\frac{\th}{2}}
        \Big\||\F [e^{-i\frac{|x|^2}{2Y(t)}}U_Y^{-1}f]|^{p-1}
         \Big\|_{L^{\frac{n}{1-\th}}}  
         \Big\||\N|^\th 
                   (\F [e^{-i\frac{|x|^2}{2Y(t)}}U_Y^{-1}]f))
         \Big\|_{L^q} \\
    \le&\,
               C |Y(t) |^{-\frac{\th}{2}}
        \Big\|\F [e^{-i\frac{|x|^2}{2Y(t)}}U_Y^{-1}f]
         \Big\|^{p-1}_{L^{\frac{n(p-1)}{1-\th}}}  
         \Big\||\N|
                   (\F [e^{-i\frac{|x|^2}{2Y(t)}}U_Y^{-1}]f))
         \Big\|_{L^2},
\end{align*}
where
we apply the fractional chain rule (see Christ-Weinstein  \cite{CW-JFA-1991} for details 
and Kenig-Ponce-Vega \cite{KPV-CPAM-1993}) to the above second estimate
and  the Gagliardo-Nirenberg inequality and the Sobolev embedding with
\begin{align*}
 \frac{1}{q}-\frac{\th}{n} =\frac{1}{2}-\frac{1}{n}, \quad 
  \frac{n(p-1)}{1-\th} = \frac{2n}{n-2},
\end{align*}
respectively.
Noting that $J(t) f=U_Y |x| U_Y^{-1}f$,
and using interpolation, we have \eqref{eqn;R-decay2}.

We next estimate $R_2(t)$ in the similar argument for $R_1$ that 
\begin{align}
  \|R_2(t)\|_{L^2} 
  \le&\,
        C \Big(
               \|\F e^{-i\frac{|x|^2}{2Y(t)}}U^{-1}_Yf
               \|^{p-1}_{L^{\frac{n(p-1)}{1-\th}}}
               +\|\F U^{-1}_Yf\|^{p-1}_{L^{\frac{n(p-1)}{1-\th}}}
            \Big) \|\F(e^{-i\frac{|x|^2}{2Y(t)}}-1)U_Y^{-1}f\|_{L^q}
        \label{eqn;R_2-est}\\
  \le&\,C\|J(t)f\|^{p-1}_{L^2}
                \| |\N|^{1-\th}
                   \F(e^{-i\frac{|x|^2}{2Y(t)}}-1)U^{-1}_Yf\|_{L^2} \nn\\
  \le&\,C|Y(t)|^{-\frac{\th}{2}}
                 \|J(t)f\|^{p-1}_{L^2}
                \||x| U^{-1}_Yf\|_{L^2},\nn
\end{align}
where $\th=\frac{n}{2}-\frac{n-2}{2}p$.
By the interpolation, we obtain \eqref{eqn;R-decay2}.
Thus, the estiamte \eqref{eqn;R-decay2}  follows for any $n \ge 3$.

Let $n=2$.
By the fractional chain rule, we have that
\begin{align*}
\|R_1(t)\|_{L^2} 
  \le&\,C|Y(t) |^{-\frac{\th}{2}}
        \Big\||\N|^\th
                \Big(|\F [e^{-i\frac{|x|^2}{2Y(t)}}U_Y^{-1}f]|^{p-1}
                   (\F [e^{-i\frac{|x|^2}{2Y(t)}}U_Y^{-1}]f)\Big)
         \Big\|_{L^2}\\
    \le&\, C|Y(t) |^{-\frac{\th}{2}}
        \Big\||\F [e^{-i\frac{|x|^2}{2Y(t)}}U_Y^{-1}f]|^{p-1}
         \Big\|_{L^{\frac{2}{1-\th}}}  
         \Big\||\N|^\th 
                   (\F [e^{-i\frac{|x|^2}{2Y(t)}}U_Y^{-1}]f))
         \Big\|_{L^q} \\
    \le&\,
               C|Y(t) |^{-\frac{\th}{2}}
        \Big\|\F [e^{-i\frac{|x|^2}{2Y(t)}}U_Y^{-1}f]
         \Big\|^{p-1}_{L^{\frac{2(p-1)}{1-\th}}}  
         \Big\||\N|
                   (\F [e^{-i\frac{|x|^2}{2Y(t)}}U_Y^{-1}]f))
         \Big\|_{L^2},
\end{align*}
where we use the H\"older inequality, the Gagliardo-Nirenberg inequality with
\begin{align*}
  \frac12=\frac{1-\th}{2}+\frac1q, 
     \quad
  \frac{1-\th}{2(p-1)}=\frac{1-\th_1}{2} + \th_1\Big(\frac{1}{2}-\frac{1}{2}\Big), 
     \quad
 \frac{1}{q}-\frac{\th}{2} 
    =\frac{1-\th_2}{2}+\th_2 \Big(\frac12-\frac{1}{2}\Big)
\end{align*}
for sufficiently small $\ep>0$.
We remark that $\th_2=1$ and $\th_1=(p+\th-2)/(p-1) < 1$ for $\th < 1$. 
 The similar argument of \eqref{eqn;R_2-est} yields 
 $\|R_2(t)\|_{L^2} \le C|Y(t)|^{-\frac{\th}{2}}\|J(t)f\|^p_{L^2}$
 for any $0<\th<1$.
 Hence we obtain \eqref{eqn;R-decay2} for $n=2$.

Let $n=1$. Then the Leibnitz rule and  the Sobolev embedding
$H^1 (\R) \subset  L^\infty(\R) $ implies 
\begin{align*}
\|R_1(t)\|_{L^2} 
  \le&\,C|Y(t) |^{-\frac{1}{2}}
        \Big\||\N|
                \Big(|\F [e^{-i\frac{|x|^2}{2Y(t)}}U_Y^{-1}f]|^{p-1}
                   (\F [e^{-i\frac{|x|^2}{2Y(t)}}U_Y^{-1}]f)\Big)
         \Big\|_{L^2}\\
    \le&\,
               C |Y(t) |^{-\frac{1}{2}}
        \Big\|\F [e^{-i\frac{|x|^2}{2Y(t)}}U_Y^{-1}f]
         \Big\|^{p-1}_{L^{\infty}}  
         \Big\||\N|
                   (\F [e^{-i\frac{|x|^2}{2Y(t)}}U_Y^{-1}]f))
         \Big\|_{L^2}\\
    \le&\,C |Y(t) |^{-\frac{1}{2}}(\|f\|_{L^2}+\|J(t)f\|_{L^2})^p,
\end{align*}
where we apply the a priori bound \eqref{eqn;bound}
to have the last estimate.
Similarly we have 
$\|R_2(t)\|_{L^2} \le C|Y(t) |^{-\frac12}(\|f\|_{L^2}+\|J(t)f\|_{L^2})^p$
and hence we obtain \eqref{eqn;R-decay2} for $n=1$.
\end{proof}

\section{Global existence of dissipative solutions}
The argument of obtaining a global solution to \eqref{eqn;NLSv} is similar to those 
of the previous works \cite{HN-AJM-1998}, \cite{K-L-S-DIE-2014}
and \cite{S-CPDE-2006}.
For arbitrary small fixed $\ep_1 > 0$, we define for $ T_0 \leq T_1 \leq 2T_0$, 
\begin{align}
\label{eqn;X_{T_1}}
\|v\|_{X_{T_1}}
  \equiv \sup_{T_0 \leq t \leq  T_1}
    \Big\{\J{Y (t)}^{\frac{n}{2}}\|v(t)\|_{L^\infty} 
      +\J{ Y(t) }^{-\ep_1}
      \Big(\|v(t)\|_{{H}^s}+\||J|^s(t)v\|_{L^2}\Big)
    \Big\},
\end{align}
where $\J{x}=(1+|x|^2)^{\frac{1}{2}}$, $Y(t)=\frac{y_2(t)}{2 y_1(t)}$ and $\ep_1 >0$ is the small constant.
We remark that $Y(t)^{-1}$ has a time decay property by the assumption (C).
\begin{lem}\label{lem;a priori}
Let $v$ be the solution to 
the problem \eqref{eqn;NLSv}. Suppose that there exists $0< \ep \ll 1 $ such that $\hat{C}_j (T_1) \ep _0 \leq \ep /3$, where $\hat{C}_j (T_1)$, $j=0,1,2$ are the equivalent to those in \eqref{5}--\eqref{8}.
Then, there exists a constant $C>0$, 
independent of $T_1$ such that
\begin{align}
\label{eqn;bound}
    \|v\|_{X_{T_1}}  \le C\|v_0\|_{H^{s,s}},
\end{align}
where 
$
     \|v_0\|_{H^{s,s}}=\|v_0\|_{H^s}+\||x|^su_0\|_{L^2} 
$
is sufficiently small and we denote $\ep_0 \equiv \|v_0\|_{H^{s,s}}$.
\end{lem}

\begin{proof}[Proof of Lemma \ref{lem;a priori}]
To prove the estimate \eqref{eqn;bound},
it suffices to show that there exists a constant $C>0$ such that
\begin{align}
\label{eqn;a priori}
    \|v\|_{X_{T_1}} 
        \le  C\|v_0\|_{H^{s,s}} + C\|v\|_{X_{T_1}}^{p}.
\end{align}
The estimate \eqref{eqn;a priori} yields the uniform bound \eqref{eqn;bound}
if we restrict the size of initial data. 
The existence of the global solution to \eqref{eqn;NLSv}
is an immediate consequence
of the a priori bound \eqref{eqn;bound}.
The  local existence theorem for the problem \eqref{eqn;dNLS}
proved by Kawamoto-Muramatsu \cite{KM-JEE-2021}
(cf. \cite{G-V-JFA-1979},  \cite{Y-CMP-1987}, \cite{C-W-MM-1988} 
and \cite{K-T-AJM-1998}).
We estimate
$\||\N|^sv(t)\|_{L^2}$ and $\||J|^s(t)u(t)\|_{L^2}$
by the energy method  due to  Hayashi-Naumkin \cite{HN-AJM-1998}
(see also \cite{K-L-S-DIE-2014}, \cite{K-AA-2016}).
We prove the a priori bound in \eqref{eqn;a priori}. 
Here we only consider the case where $p$ is critical in the sense of $y_2(t)$, because the super-critical case can be shown more easily. In the followings, we employ the Duhamel formula; 
\begin{align*}
v(t) &= U(t,T_0) v_0 -i \lambda \int_{T_0}^t U(t,\t)  |u (\t)|^{p-1} u(\t) d\t
\\ & 
= U_Y(t) U_Y(T_0)^{-1} v_0 -i \lambda \int_{T_0}^t U_Y(t) U_Y (\t)^{-1}  |u (\t)|^{p-1} u(\t) d\t.
\end{align*} 
We first estimate the derivative term $\|v(t)\|_{\dot{H}^s}$ such that for any $t \in [T_0, T_1]$,
\begin{align*}
\nn \|v(t)\|_{\dot{H}^s}
 \le& \,\|v(T_0)\|_{\dot{H}^s}
    +|\lam|
   \int_{T_0}^t \left| y_1 (\tau) \right|^{-\frac{n(p-1)}{2}} 
       \||v(\tau) |^{p-1} v(\tau)\|^{p}_{\dot{H}^s}\, 
   d\tau\\
\nn
  \le&\, C\|v_0\|_{\dot{H}^s} + C
      \int_{T_0}^t \left| y_1 (\tau) \right|^{-\frac{n(p-1)}{2}} 
            \|v(\tau)\|^{p-1}_{L^\infty}\|v(\tau)\|_{\dot{H}^s}\,d\tau\\
  \le&\,C\|v_0\|_{\dot{H}^s} + C{\|v\|_{X_{T_1}}^{p-1}}
      \int_{T_0}^t \left| y_1 (\tau) \right|^{-\frac{n(p-1)}{2}}  
          \J{Y(\tau)}^{-\frac{n(p-1)}{2}} \|v(\tau)\|_{\dot{H}^s}\,d\tau \\
        \nn
    \le&\,C\|v_0\|_{\dot{H}^s} + C{\|v\|_{X_{T_1}}^{p-1}}
        \int_{T_0}^t 
           \left| y_2 (\tau) \right|^{-\frac{n(p-1)}{2}}  
            \|v(\tau)\|_{\dot{H}^s}\,
       d\tau,
\end{align*}
where $Y(t)=y_2(t)/2y_1(t)$ and we use $\| U_Yu \|_{\dot{H}^s} = \| u \|_{\dot{H}^s}$.
The definition \eqref{eqn;X_{T_1}} and \eqref{eqn;critical} yields 
\begin{align}
    \left\| v(t) \right\|_{\dot{H}^s} 
          \leq \|v_0\|_{\dot{H}^s} 
              + \left\| v \right\|^p_{X_{T_1}} \J{Y(t)}^{\ep_1}.
    \label{eqn;del-X_1}
\end{align}

We next estimate as $\| |J|^s(t)v(t)\|_{L^2}$. Using the commutative relation $[i\pt_t + \frac{1}{2y_1^2}\pt^2_x, |J|^s(t)]=0$ 
and the expression of the operator $|J|^s(t)$ such that
$
|J|^s(t)=U_Y(t)|x|^s U_Y(t)^{-1},
$ 
we also have
\begin{align*}
\| |J|^s(t)v(t)\|_{L^2}
\le&\, \| |x|^s U_Y(T_0)^{-1} v_0 \|_{L^2} +|\lam|
   \int_{T_0}^t \left| y_1 (\tau) \right|^{-\frac{n(p-1)}{2}} 
       \||x|^s U_Y(\tau )^{-1}|v(\tau) |^{p-1} v(\tau)\|_{L^2}\, 
   d\tau \nonumber \\
   \le&\,C\|v_0\|_{H^{s,s}}+ C{\|v\|_{X_{T_1}}^{p-1}}
              \int_{T_0}^t 
                    \left| y_1 (\tau) \right|^{-\frac{n(p-1)}{2}}  
                    \J{Y(\tau)}^{-\frac{n}{2}(p-1)} \||J|^s(\tau)v(\tau)\|_{L^2}\,
               d\tau  \\
     \le&\,\| v_0\|_{H^{s,s}}+ C{\|v\|_{X_{T_1}}^{p-1}}
              \int_{T_0}^t 
                    \left| y_2 (\tau) \right|^{-\frac{n(p-1)}{2}}  
                    \||J|^s(\tau)v(\tau)\|_{L^2}\,
               d\tau \nonumber 
 \end{align*}
The definition \eqref{eqn;X_{T_1}} and \eqref{eqn;critical} yields 
 \begin{align}
 \| |J|^s(t)v(t)\|_{L^2} 
     \leq  \left\|  v_0 \right\|_{H^{s,s}}
            + \left\| v \right\|^p_{X_{T_1}} \J{Y(t)}^{\ep_1}. 
     \label{eqn;J-X_1}
 \end{align}
 
 Finally, we consider the $L^\infty$-bound 
of the solution to \eqref{eqn;NLSv}:
 \begin{align}
 \label{eqn;|v|-t-geq-1}
  \J{y_2(t)}^{\frac{n}{2}}
       \|v(t)\|_{L^\infty}
           \le C\|v_0\|_{H^{s,s}}
               +C\|v\|_{X_{T_1}}^p.
\end{align}
To this end,
we first show the uniform bound of 
$$\tilde{v} \equiv \F[U_Y(t)^{-1} v]. $$
Namely, for any $t \in [T_0,T_1]$, we show
\begin{align}
\label{eqn;unif-apri2}
  \|\tilde{v}(t)\|_{L^\infty} 
     \le C\|v_0\|_{H^{s,s}}  
           +C \|v\|^{p}_{X_{T_1}}.
\end{align}
If we show the uniform bound \eqref{eqn;unif-apri2},
we obtain the estimate \eqref{eqn;|v|-t-geq-1}
by applying the Fourier transform to \eqref{eqn;NLSv}
in such a way that
\begin{align}
\label{eqn;anal-profile} 
i\pt_t \tilde{v}(t,\xi) 
   =
   \,& \F \Big[U_Y(t)^{-1}
                  \Big(i\pt_t v + \frac{1}{2y_1^2}\Del v\Big)
               \Big]
     =\,\frac{\lam}{|y_2(t)|^{\frac{n}{2}(p-1)}}\F \Big[U_Y(t) ^{-1}
              \Big(
                  |v|^{p-1} v
             \Big)  \Big]\\
\nn    =& \frac{\lam}{|y_2(t) |^{\frac{n}{2}(p-1)}}\,
          |\tilde{v}(t,\xi)|^{p-1}
          \tilde{v}(t,\xi) + R(t,\xi),
\end{align}
where $R$ is a remainder term given by
\begin{align}\label{eqn;R}
R(t)=\frac{\lam}{|y_2(t)|^{\frac{n}{2}(p-1)}}\,\F \Big[U_Y(t) ^{-1}
                 \Big(
                       {|v|^{p-1}v)}
                 \Big)         \Big]
              - \frac{\lam}{|y_2(t)|^{\frac{n}{2}(p-1)}}
                   |\tilde{v}|^{p-1}\tilde{v}.
\end{align}
Multiplying the both sides of the equation \eqref{eqn;anal-profile} 
by $\overline{\tilde{v}}$
and taking the imaginary part, we obtain
\begin{align}
\label{eqn;V-W}
 \frac{1}{2}\pt_t|\tilde{v}(t,\xi)|^2
          =&\,\frac{\text{Im}\,\lam}{|y_2(t)|^{\frac{n}{2}(p-1)}}
                     |\tilde{v}(t,\xi)|^{p+1}
               + \text{Im}\,R(t,\xi)(\overline{\tilde{v}(t,\xi)}).
\end{align}
From the condition $\text{Im}\, \lam < 0$, 
we see that the first term 
of the right hand side
in \eqref{eqn;V-W} are non-positive. 
Therefore, by noting $\partial _t |f(t) |^2 = 2 |f(t)| \partial _t |f (t)|  $ for $f(t) \in C^1 (\mathbb{R}\,;\,\mathbb{C})$, we obtain 
\begin{align}\label{eqn;V}
  \frac{1}{2}
      \pt_t|\tilde{v}(t,\xi)|
          \le  |R(t,\xi)|, 
       \quad T_0 \leq t \leq T_1,\,\,\xi \in \re^n.
\end{align}
Integrating the both sides of the inequality \eqref{eqn;V}
over $[T_0,t]$, $t \leq  T_1$, we see that
\begin{align}
\label{eqn;int-V}
     |\tilde{v}(t,\xi)|
           \le \,|\tilde{v}(T_0,\xi)|
               + 2\int_{T_0}^{T_1}
                      |R(\tau,\xi)|\, 
                 d\tau.
\end{align}

By applying the estimate \eqref{eqn;R-decay1} in Lemma \ref{lem;R-decay} 
to the remainder term $R(t,\xi)$, 
we see that for any $t \in [T_0,T_1],\,\,\xi \in \re ^n$,
\begin{align}\label{eqn;R-e^y}
|R(t,\xi)| \le C | y_2(t) |^{-\frac{n}{2}(p-1)} |Y(t)|^{-s_1} 
     (\|v(t)\|_{L^2}+\||J|^s(t)v(t)\|_{L^2})^{p},
\end{align}
where $0< s_1 < \min \{ 1, s -n/2 \} $ and $Y(t)=y_2(t)/2y_1(t)$. 
Combining \eqref{eqn;R-e^y}, \eqref{eqn;J-X_1},
and the $L^2$-a priori bound; 
\begin{align}
\label{eqn;mass}
  \|v(t)\|^2_{L^2} 
       + 2|{\rm Im}\,\lam|  
       \int_0^t |y_1(\t)|^{-\frac{n(p-1)}{2}}\|v(\t)\|^{p+1}_{L^{p+1}}\,d\t
     = \|v_0\|^2_{L^2},
\end{align}
we obtain that for arbitrary small $\ep_1>0$ and $t \geq T_0$,
\begin{align}
\label{eqn;R_1-R_2}
\|R(t)\|_{L^\infty}
   \le\,C|y_2(t)|^{-\frac{n}{2}(p-1)} 
              |Y(t)|^{-s_1} \J{Y(t)}^{p\ep_1}
         \|v\|^{p}_{X_{T_1}}.
\end{align}
There exists $\ep_2>0$ such that for small $\ep_1$,
it follows
 $|Y(t)|^{-s_1} \J{Y(t)}^{p\ep_1} \le Ct^{- \delta (s_1 -p \ep_1)}  \le Ct^{-\ep_2} $, 
where we use \eqref{K10}, and hence
the estimate \eqref{eqn;int-V} yields by \eqref{eqn;R_1-R_2}
and the embedding $H^s (\mathbb{R}^n) \subset L^\infty (\mathbb{R} ^n)$ for $s> n/2$ that
\begin{align}
\label{eqn;V-L^infty2}
\left\| \tilde{v}(t) \right\|_{L^{\infty}} 
   &\leq C \left\| v(T_0)\right\|_{H^s} 
        +C \|v\|^{p}_{X_{T_1}} 
           \int_{T_0}^t 
                |y_2(\tau)|^{-\frac{n}{2}(p-1)} \t^{-\ep_2}
          d \tau 
\\& \leq C \|v_0\|_{H^{s,s}} + C \| v \|_{X_{T_1}}^{p}, \nonumber
\end{align}
where the constant $C>0$ is not depend on $T_1$. Then we have by $U_Y(t) = \CAL{M}(Y(t) ) \CAL{D}(Y(t)) \CAL{F} \CAL{M}(Y(t))$  in \eqref{eqn;MDFM} that
\begin{align}
\left\| v(t) \right\|_{L^{\infty}} 
   &=\left\| U_Y(t) U_Y(t)^{-1} v(t)\right\|_{L^{\infty}} \nonumber\\
   &= |Y (t)|^{-n/2} 
           \left\| 
               \CAL{F}[ \CAL{M}(Y(t)) U_Y(t)^{-1}] v(t) 
           \right\|_{L^{\infty}} \label{eqn;point-uv}\\
  & \le |Y(t)|^{-n/2}
        \left(
            \left\| \CAL{F}U_Y(t)^{-1} v(t) \right\|_{L^{\infty}} 
        + \left\|
                \CAL{F} [(\CAL{M}(Y(t))-1) U_Y(t)^{-1} v](t) 
          \right\|_{L^{\infty}} 
       \right).  \nonumber
\end{align}
Using 
$ 
   \| ( \CAL{M}(Y(t))  -1)f \|_{L^1} 
      \leq C |Y(t)|^{-s_1} \| \J{x}^{s} f\|_{L^2} 
$, 
we obtain 
\begin{align}
  \| v(t) \|_{L^\infty} 
       \leq C |Y(t)|^{-n/2} 
       \left(  \|v_0\|_{H^{s,s}} + |Y(t)|^{- (s_1 -p \ep_1)} \| v \|_{X_{T_1}}^{p} \right).
   \label{eqn;point-decay}
\end{align}
Therefore, we obtain the $L^\infty$-estimate \eqref{eqn;|v|-t-geq-1}.

From inequalities \eqref{eqn;del-X_1},
\eqref{eqn;J-X_1} and \eqref{eqn;|v|-t-geq-1},
the a priori bound \eqref{eqn;a priori} holds and
we obtain \eqref{eqn;bound}
for sufficiently small initial data.
\end{proof}

\begin{lem}\label{lem;L^2-a priori}
Let $u$ be the solution to \eqref{eqn;NLS-y_1} satisfying \eqref{eqn;bound}.
Then, there exists $C>0$ such that for any $t \geq T_0$, 
\begin{align}\label{eqn;H^m-apri}
\|v(t)\|_{\dot{H}^{\frac{s}{2}}} \le C\|v_0\|_{{H}^{s,s}}.
\end{align}
\end{lem}

\begin{proof}[Proof of Lemma \ref{lem;L^2-a priori}.]
Multiplying \eqref{eqn;V-W} by $|\xi|^s$ 
and integrating \eqref{eqn;V-W} over $ [T_0,t] $, 
we have that
\begin{align*}
\frac{1}{2} \left| |\xi|^{\frac{s}{2}} \tilde{v} (t, \xi)  \right|^2 \leq \frac{1}{2} \left| |\xi|^{\frac{s}{2}} \CAL{F}^{} U_Y(T_0)^{-1} v_0  \right|^2 + \int_{T_0}^t \left| R(t, \xi) |\xi|^s \overline{\tilde{v}(\t , \xi)} \right| d \tau
\end{align*}
and integrating it over $\R$ in $\xi$ and using H\"{o}lder's inequality, we have
\begin{align}
\label{eqn;m-L^2-2}
       \||\xi|^{\frac{s}{2}}\widehat{v}(t)\|_{L^2}^2
           \le \,\|v_0\|_{\dot{H}^{\frac{s}{2}}}^2
               +2\int_{T_0}^t
                      \|R(\tau)\|_{L^{2}}\||\xi|^{s}\widehat{v}(\t)\|_{L^2}\,
                 d\tau,
\end{align}
where we use 
\begin{align*}
\left\| |\xi|^{\frac{s}{2}} \tilde{v} (t, \xi)  \right\|_{L^2_{\xi}} = \left\| e^{-i \frac{y_2(t)}{2y_1(t)} |\xi|^2 }|\xi|^{\frac{s}{2}} \left[ \CAL{F} {v } \right] (t, \xi)  \right\|_{L^2_{\xi}} = \left\| |\xi|^{\frac{s}{2}} \left[ \CAL{F} {v} \right] (t, \xi)  \right\|_{L^2_{\xi}}.
\end{align*}
By the estimate \eqref{eqn;bound}, it follows
\begin{align}\label{eqn;s/2-L^2}
 \||\xi|^{s}\widehat{v}(t)\|_{L^2}
   \le C|Y(t)|^{\ep_1}\|v\|_{X_{T_1}} \le C|Y(t)|^{\ep_1}\|v_0\|_{H^{s,s}}
\end{align}
and \eqref{eqn;R-decay2} yields the integrability of 
\begin{align*}
\|R(\tau)\|_{L^{2}}\||\xi|^{s}\widehat{v}(\t)\|_{L^2} \leq C |y_2(\t)|^{-\frac{n}{2}(p-1) } \J{Y(\t)}^{-(\frac12 - 2\ep_1)} \| v_0 \|_{H^{s,s}}.
\end{align*}
Hence, by \eqref{eqn;s/2-L^2}, we obtain the uniform bound \eqref{eqn;H^m-apri}.
\end{proof}

\section{Proof of the main Theorems}
In this section, we prove $L^2$-decay of dissipative solutions to \eqref{eqn;dNLS}.
In our proof, we employ the previous approach by \cite{H-L-N-AMP-2016} besides the frequency dividing approach due to \cite{O-S-NoDEA-2020}. 
Such approach enable us to remove an extra exponent $\ep>0$ appeared in $L^2$-decay order in \cite{H-L-N-AMP-2016}.
By applying Lemma \ref{lem;L^2-a priori}, 
one can extract the $L^2$-decay of dissipative solutions to \eqref{eqn;dNLS}.
We have that this $L^2$-decay holds 
for higher dimensions $n=1,2,3,$
even if we only assume Im $\lam<0$.
Moreover this $L^2$-decay property can be seen for the problem \eqref{eqn;dNLS} which contains the time dependent potential.
\begin{proof}[Proof of Theorem \ref{thm;L^2-decay}]
Let $1 \le n \le 3$, $p >1$ be the critical or super-critical and $n/2<s<p$.
We first show the pointwise decay \eqref{eqn;Point-decay}.
By \eqref{eqn;R-decay1}, 
the error term \eqref{eqn;R} is estimated by
\begin{align}
\label{eqn;R-decay}
  \|R(t)\|_{L^\infty} 
        \le C |y_2(t)|^{-\frac{n}{2}(p-1)} |Y(t)|^{-(s_1 -p \ep_1)}\| v_0 \|_{{H}^{s,s}}^p , 
      \quad t \geq T_0,
\end{align}
where $Y(t)=y_2(t)/2y_1(t)$.
Let $Y_2(t)=\int_{T_0}^t |y_2(\t)|^{-\frac{n}{2}(p-1)}\,d\t$
and $\tilde{v}(t,\xi)=\F [U_Y(t)^{-1}v](t,\xi)$.
Then, the differential equation \eqref{eqn;V-W} and the inequality 
\begin{align*}
\left| 
R(t) \overline{\tilde{v}(t,\xi)}
\right| &\leq C|y_2 (t) | ^{- \frac{n}{2} (p-1) } \J{Y(t)}^{-(s_1-p \ep_1)} \| v_0 \|_{H^{s,s}} \left( \left\| v_0 \right\|_{H^{s,s}} + \left\| v \right\|_{X_{T_1}}^p \right) 
\\ & \leq C|y_2 (t) | ^{- \frac{n}{2} (p-1) } \J{Y(t)}^{-(s_1-p \ep_1)} \| v_0 \|_{H^{s,s}}^{p+1} 
\end{align*} 
lead to
\begin{align*}
&  \frac{d}{dt}\Big(Y_2(t)^{\frac{2}{p-1}+1}|\tilde{v}(t,\xi)|^2\Big)
    \\  =&\, \Big(\frac{d}{dt}Y_2(t)^{\frac{2}{p-1}+1}\Big)
          |\tilde{v}(t,\xi)|^2
        + Y_2(t)^{\frac{2}{p-1}+1} \frac{d}{dt} |\tilde{v}(t,\xi)|^2
        \nonumber \\
    \le&\,\,C|y_2(t)|^{-\frac{n}{2}(p-1)} 
                    Y_2(t)^{\frac{2}{p-1}}
          |\tilde{v}(t,\xi)|^2 \\
        &+ Y_2(t)^{\frac{2}{p-1}+1} 
           \bigg(
                \frac{-|\text{Im}\,\lam|}{|y_2(t)|^{\frac{n}{2}(p-1)}}
                     |\tilde{v}(t,\xi)|^{p+1}
               + C|y_2(t)|^{-\frac{n(p-1)}{2}}|Y(t)|^{-(s_1 -p\ep_1)} \left\| v_0 \right\|_{H^{s,s}}^{p+1}
           \bigg).
      \nonumber
\end{align*}
By the Young inequality with $\frac{p-1}{p+1}+\frac{2}{p+1}=1$, there exists $C>0$ such that for small $\ep>0$,
\begin{align*}
    |\tilde{v}(t,\xi)|^2 = \left( \ep {Y_2(t)} \right)^{- \frac{2}{p+1}}  \cdot \left( \ep {Y_2(t)} \right)^{ \frac{2}{p+1}}|\tilde{v}(t,\xi)|^2
       \le  C\ep^{-\frac{2}{p-1}} Y_2(t)^{-\frac{2}{p-1}}
           + C\ep^{} Y_2(t) |\tilde{v}(t,\xi)|^{p+1}.
\end{align*}
Hence, we have by taking $\ep = |\mathrm{Im} \lambda| \ep_3$ with small $\ep _3 >0$ that
\begin{align}
  \frac{d}{dt}\Big(Y_2(t)^{\frac{2}{p-1}+1}|\tilde{v}(t,\xi)|^2\Big)
     \le&\, C(\ep_3 |\mathrm{Im} \lambda|)^{-\frac{2}{p-1}}
                   |y_2(t)|^{-\frac{n(p-1)}{2}} \nonumber\\
        &+ CY_2(t)^{\frac{2}{p-1}+1} 
              |y_2(t)|^{-\frac{n(p-1)}{2}} |Y(t)|^{-(s_1 -p\ep_1)} \left\| v_0 \right\|_{H^{s,s}}^{p+1}
      \label{eqn;bibun1}
\end{align}
and integrating \eqref{eqn;bibun1} over $[T_0,t]$, we deduce that
\begin{align*}
    Y_2(t)^{\frac{2}{p-1}+1}|\tilde{v}(t,\xi)|^2
     \le&\, C(\ep_3 |\mathrm{Im}  \lambda |)^{-\frac{2}{p-1}} Y_2(t)\\
        &+ C \left\| v_0 \right\|_{H^{s,s}}^{p+1} 
        \int_{T_0}^t  Y_2(\t)^{\frac{2}{p-1}+1} 
                    |y_2(\t)|^{-\frac{n(p-1)}{2}} |Y(\t)|^{-(s_1 -p\ep_1)} 
              \,d\t \nonumber\\
       \le&\,C(\ep_3 |\mathrm{Im} \lambda |)^{-\frac{2}{p-1}} Y_2(t)
             + C \left\| v_0 \right\|_{H^{s,s}}^{p+1},
\end{align*}
where the second  part is integrable since $p$ is the critical and the assumption (C) holds.
Thus, we have that for any $t \ge T_0$,
\begin{align}
\label{eqn;po-decay}
  \|\tilde{v}(t)\|_{L^\infty}^2
     \le C \left( (\ep_3  | \mathrm{Im} \lambda |)^{-\frac{2}{p-1}} + Y_2(t)^{-1} \left\| v_0 \right\|_{H^{s,s}}^{p+1}\right) Y_2(t)^{-\frac{2}{p-1}}
\end{align}
and we obtain by  \eqref{eqn;point-uv} and $\|u(t)\|_{L^\infty}=|y_1|^{-\frac{n}{2}}\|v(t)\|_{L^\infty}$ that
\begin{align*}
\nn   \|u(t)\|_{L^{\infty}} 
     \le &\, C|y_1(t)Y(t)|^{-\frac{n}{2}}
           (\|\tilde{v}(t)\|_{L^\infty} + |Y(t)|^{-(s_1-p\ep_1)}\|v\|_{X_{T_1}}^p)\\
      \le &\,C|y_2(t)|^{-\frac{n}{2}} Y_2(t)^{- \frac{1}{p-1}}
          \left( ( \ep_3 |\mathrm{Im} \lambda |)^{-\frac{1}{p-1}} 
          + Y_2(t)^{-1} \left\| v_0 \right\|_{H^{s,s}}^{p+1}  \right)^{1/2} \\
\nn          &+ |y_2(t)|^{- \frac{n}{2}}|Y(t)|^{-(s_1-p \ep_1)} 
          \left\| v_0 \right\|_{H^{s,s}}^{p} ,
\end{align*}
where $|Y(t)|=|y_2(t)/2y_1(t)| \le C t^{-\del}$ and $\ep_2>0$ is the same one in \eqref{eqn;V-L^infty2}.

\vspace{2mm}
We prove the $L^2$-decay of solutions to \eqref{eqn;dNLS} with higher regular condition $s > n/2$.
We decompose the  low-frequency and high frequency part of the solution.
The low frequency part is controlled by the pointwise estimate \eqref{eqn;po-decay}
and high frequency part is estimated by the uniform estimate.
Namely, the solution $u$ satisfies that for any $t \ge T_0$ and $r>0$, 
\begin{align*}
  \|u(t)\|^2_{L^2}
     =&\, \|\tilde{v}(t)\|^2_{L^2(|\xi|\le r)}
         +\|\tilde{v}(t)\|^2_{L^2(|\xi|>r)} \nonumber\\
         \leq & Cr^n \left\| \tilde{v}(t) \right\|^2_{L^{\infty}} + Cr^{-s} \left\| |\xi|^{\frac{s}{2}} \tilde{v}(t)  \right\|^2_{L^2(|\xi| > r)}
         \\
     \le&\,Cr^nY_2(t)^{-\frac{2}{p-1}}\left( (\ep_3  | \mathrm{Im} \lambda |)^{-\frac{2}{p-1}} + Y_2(t)^{-1} \left\| v_0 \right\|_{H^{s,s}}^{p+1}\right)
                   + Cr^{-s}
           \||\xi|^{\frac{s}{2}}\widehat{v}(t)\|^2_{L^2}
           \\  \leq &
           Cr^nY_2(t)^{-\frac{2}{p-1}}\left( (\ep_3  | \mathrm{Im} \lambda |)^{-\frac{2}{p-1}} + Y_2(t)^{-1} \left\| v_0 \right\|_{H^{s,s}}^{p+1}\right)
                   + Cr^{-s}
           \|v_0\|^2_{H^{s,s}}
           ,
\end{align*}
where $Y_2(t)=\int_{T_0}^t |y_2(\t) |^{-\frac{n}{2}(p-1)}\,d\t$ and we use Lemma \ref{lem;L^2-a priori}. This inequality is optimized by taking 
$r=(|\mathrm{Im} \lambda |Y_2(t))^{\frac{2}{(p-1)(s+n)}} \| v_0 \|_{H^{s,s}} ^{\frac{2}{n+s}} $.
Namely, we are valid that
\begin{align*}
   \|u(t)\|_{L^2} \le C (|\mathrm{Im} \lambda |  Y_2(t))^{-\frac{s}{(p-1)(s+n)}}  \| v_0 \|_{H^{s,s}}^{\frac{n}{n+s}} + C|\mathrm{Im} \lambda| ^{\frac{n}{(p-1) (s+n)}} Y_2(t)^{-\frac{s}{(p-1) (s+n)} - \frac12} \| v_0 \|_{H^{s,s}}^{\frac{p+1}{2}} .
\end{align*}
We next consider the uniform lower bound when $p$ is the super-critical case.
Let $Y(t)=\frac{y_2(t)}{2y_1(t)}$,
$J(t)f=(x+iY(t)\N)f$, $U(t,T_0) \equiv U_Y(t)^{-1} U_Y(T_0) $, $U_Y(t) = e^{ i Y(t) \Delta }$
and let $v$ be the solution to \eqref{eqn;NLSv} 
with $v_0 \in H^{s,s}(\re^n)$
given via a priori estimate \eqref{eqn;Upper-bound}.
Since the pointwise estimate 
\begin{align}
\label{eqn;infty-est}
  \|v(t)\|_{L^\infty} 
   \le C\|v_0\|_{H^{s,s}} |Y(t) |^{-\frac{n}{2}}
     +C\|v_0\|_{H^{s,s}}^{p} |Y(t) |^{-\frac{n}{2}}
        |Y(t) |^{-(s_1-p\ep_1)}
\end{align}
holds, one can apply the similar argument in \cite{KS-AA-2021}. 
we deduce by combining \eqref{eqn;infty-est} and the 
$L^2$-dissipative identity:
\begin{align}\label{eqn;mass2}
\|v(t)\|^2_{L^2}  
    = \|v_0\|^2_{L^2}
           - |{\rm Im}\,\lam|
                   \int_0^t
                        |y_1(\t)|^{-\frac{n}{2}(p-1)}
                        \|v(\tau)\|^{p+1}_{L^{p+1}}\,
                   d\tau, 
    \quad t \geq 0,
\end{align}
that
\begin{align}
\label{eqn;L1}
  \frac{d}{dt}\|v(t)\|_{L^2}^2
        \ge&\, 
               -|{\rm Im}\,\lam| |y_1(t)|^{-\frac{n}{2}(p-1)} 
                           \|v(t)\|^{p-1}_{L^\infty}
                         \|v(t)\|_{L^2}^{2}\\
  \nn      \geq&\,
     - C |\mathrm{Im} \lambda| |y_2(t)|^{- \frac{n}{2}(p-1)} \left( \|v_0\|_{H^{s,s}}
         + t^{- \ep_2}  \|v_0\|_{H^{s,s}}^{p}   \right)
        \|v(t)\|^{2}_{L^{2}}.
\end{align}
By solving \eqref{eqn;L1}, one can obtain the $L^2$-lower bound for small solutions to \eqref{eqn;NLSv} since $y_2$ satisfies \eqref{eqn;sg-equation} and
since $\|v(t)\|_{L^2}=\|u(t)\|_{L^2}$, we obtain the $L^2$-lower bound \eqref{eqn;Lower-bound} for the original solution to \eqref{eqn;dNLS}.
\end{proof}

\vspace{2mm}
\begin{proof}[Proof of Theorem \ref{thm;L^2-decay2}]
We show the $L^2$-decay of solutions for $n \geq 1$ or under the low regularities for $u_0$. By the low regular condition $s=1$, the pointwise estimate \eqref{eqn;po-decay} does not hold for the higher spatial dimensions $n \ge 2$.
Then, one can not prove the $L^2$-decay estimate in the same way of the proof of Theorem \ref{thm;L^2-decay}.
Hence, we need to improve the estimate of the low-frequency part by considering the asymptotic form of the solution to \eqref{eqn;dNLS}.
Let $n \ge 1$ and $p >1$, $\lam$ satisfy \eqref{eqn;SD} and $r>0$. Under the assumption \eqref{eqn;SD}, we find as the consequence of \cite{H-L-N-CPAA-2017} that for $t \geq T_0$,
\begin{align}
\left\| v(t) \right\|_{L^2}^2 
    + \left\| v(t) \right\|_{\dot{H}^1}^2 
        + \left\| J(t) v(t) \right\|_{L^2}^2 
  \leq   \left\| v(T_0) \right\|_{L^2}^2 
              + \left\| v(T_0) \right\|_{\dot{H}^1}^2 
         + \left\| J(T_0) v(T_0) \right\|_{L^2}^2.
   \label{eqn;unif-bound}
\end{align}
Here the problem occurs when we consider the existence and boundedness of the time-in-local solution from $t=0$ to $t= T_0$ since even the solution of the free equation $i\partial _t u = H_0(t) u$ does not have $H^1$ conservation law (conservation law of $(\zeta_2(t) D_x - \zeta _2 '(t)x )^2$ is known, see, section 5 of \cite{Ca} or Lemma 2.4 of \cite{KM-JEE-2021}). 
Let $k \in {\mathbb{N}} $ and set the energy norm 
\begin{align*}
\left\|  u  \right\|_{Y_{kT'}}
   \equiv \sup_{(k-1) T' \leq t \leq k T'} 
      \left( 
            \left\| u(t) \right\|^2_{\dot{H}^1} 
            + \left\| u(t) \right\|^2_{H^{0,1}}
     \right)^{\frac{1}{2}},
\end{align*}
where $T'$ is a small positive constant so that $0<T' < 1/64$ and
\begin{align*}
T'  \sup_{\t \in \R}|\sg(\t)| < \frac1{64}
\end{align*}
and show that for any $k \in \mathbb{N}$
\begin{align} \label{K21}
\left\| u \right\|_{Y_{kT'}} 
    \leq 4^{k-1} \left\| u_0 \right\|_{H^{1,1}}.
\end{align}
Under the assumption \eqref{eqn;SD}, we can employ the approach of \cite{H-L-N-CPAA-2017} and then the straightforward calculation shows
\begin{align*}
   \frac12 \partial _t  
    \left\| \nabla u(t) \right\|^2_{L^2} 
        &= -2 \sigma(t)  \mathrm{Im} \int_{\R ^n} 
                   \left( x \cdot \nabla u \right) \overline{u} dx
                         + \mathrm{Im} 
                    \left( \lambda 
                                \int_{\R ^n}  \nabla \left( |u|^{p-1} u \right)
                                 \cdot \overline{\nabla u}  dx
                  \right)
\end{align*}
and which deduces
\begin{align*}
   \sup_{(k-1) T' \leq t \leq k T'} 
        \left\| \nabla u(t) \right\|^2_{L^2} 
              & \leq \left\| \nabla u((k-1)T' ) \right\|^2_{L^2}  + 4 \left( \int_{(k-1)T' }^{kT'} |\sigma (\t)| d \t \right)   \left\| u\right\|^2_{Y_{k T'}}
 \\ & \leq  \left\| u \right\|^2_{Y_{(k-1) T'}}   + \frac{1}{16} 
    \left\| u\right\|^2_{Y_{k T'}}.
\end{align*}
By the similar calculations, we also get
\begin{align*}
\frac12 \partial _t 
    \left\| x u(t) \right\|^2_{L^2} 
        &= 2  \mathrm{Im} \int_{\R ^n} 
               \left( x \cdot \nabla u \right) \overline{u} dx 
         + \mathrm{Im} 
              \left( \lambda \int_{\R ^n} 
                   |x|^2|u|^{p+1}\,  dx
             \right)
\\ & \leq 2 \mathrm{Im} \int_{\R ^n} \left( x \cdot \nabla u \right) \overline{u} dx ,
\end{align*}
and that 
\begin{align*}
   \sup_{(k-1) T' \leq t \leq kT'}
    \left\| x u(t) \right\|^2_{L^2} 
          \leq \left\| u\right\|^2_{Y_{(k-1) T'}}
   + \frac1{16}
       \left\| u\right\|_{Y_{kT'}}^2.
\end{align*}
These deduce 
\begin{align}  \label{K22}
\left\| u \right\|^2_{Y_{kT'}} 
   \leq   4 \left\| u \right\|^2_{Y_{(k-1)T'}}.
\end{align}
By the standard arguments in construction mappings, we have a unique solution $u(t) \in C([ 0, T'] \, ; \, H^{1,1}  )$ by taking $k=1$, and using the induction \eqref{K22}, we have, for each $k$, a unique solution $u(t) \in C([ (k-1)T' , kT'] \, ; \, H^{1,1}  )$. Using such solutions and \eqref{K22}, the desired result \eqref{K21} can be obtained. Since we can choose $k$ arbitrarily, one can also find the existence of a global-in-time unique solution $u(t) \in C([0, \infty) \, ; \, H^{1,1}  )$ without any restrictions of the size of $u_0$. By taking $k_0 \in \mathbb{N}$ as the smallest integer so that $k_0T' > T_0 $
\begin{align*}
\left\| u \right\|^2_{Y_{T_0}} \leq 4^{k_0-1} \| u_0 \|_{H^{1,1}} \leq 4^{T_0/T'} \| u_0 \|_{H^{1,1}},
\end{align*}
where 
\begin{align*}
\left\|  u  \right\|_{Y_{T_0}}
   \equiv \sup_{0 \leq t \leq T_0} 
      \left( 
            \left\| u(t) \right\|^2_{\dot{H}^1} 
            + \left\| u(t) \right\|^2_{H^{0,1}}
     \right)^{\frac{1}{2}}.
\end{align*}
After $t \geq T_0$, one gets the uniform bounds from \eqref{eqn;unif-bound} that  
\begin{align*}
\sup_{T_0 \leq t \leq T_1}  
\left\| v(t) \right\|_{\dot{H}^1}^2 
    \leq&\, \left\| v_0 \right\|_{L^2}^2 + \left\| v(T_0) \right\|_{\dot{H}^1}^2
          + \left\| J(T_0) v(T_0) \right\|_{L^2}^2 \\
    \leq&\, C_{T_0} \left\| u (T_0) \right\|_{H^{1,1}}^2 
    \leq C_{T_0} 4^{T_0/T'} \left\| u_0 \right\|_{H^{1,1}}^2
\end{align*} 
and 
\begin{align*}
\sup_{T_0 \leq t \leq T_1}  \left\| J(t)v(t) \right\|_{L^2}^2 
    \leq&\, \left\| v_0 \right\|_{L^2}^2 + \left\| v(T_0) \right\|_{\dot{H}^1}^2 
               + \left\| J(T_0) v(T_0) \right\|_{L^2}^2 \\
    \leq&\, C_{T_0} \left\| u (T_0) \right\|_{H^{1,1}}^2 \leq C_{T_0} 4^{T_0/T'} \left\| u_0 \right\|_{H^{1,1}}^2, 
\end{align*} 
where we can choose $T_1$ arbitrarily large. Piecing these time local arguments and 
\eqref{eqn;unif-bound}, we have \eqref{eqn;Upper-bound}.

\vspace{2mm}
We next show the $L^2$-decay of dissipative solutions to \eqref{eqn;dNLS}. By the H\"older inequality, it follows that for any $r>0$,
\begin{align}
\label{eqn;L^2-L^4-prop}
  \| f\|_{L^2(|\xi|<r)} 
             \le Cr^{\frac{n(p-1)}{2(p+1)}}\|f\|_{L^{p+1}(|\xi|<r)}
    \quad \text{for any}\,\,f \in L^{p+1}(\re ^n),
\end{align}
where $p \ge 1$ and $\frac{p-1}{p+1}+\frac{2}{p+1}=1$. 
By applying Young's inequality to the estimate \eqref{eqn;L^2-L^4-prop} with 
$\frac{p-1}{p+1}+\frac{2}{p+1}=1$, we see that  
there exists $C>0$ such that for any $\ep>0$ and $y > 0$,
\begin{align}\label{eqn;young}
 \|f\|^2_{L^2(|\xi| <r)}
    \le  C\ep y \|f\|^{p+1}_{L^{p+1}(|\xi|<r)}
         +C\ep^{-\frac{2}{p-1}} y^{-\frac{2}{p-1}} r^{n}.
\end{align}
Let  $v$ be the global solution to \eqref{eqn;NLSv}
with $u_0 \in H^{1,1}(\re^n )$.
Multiplying \eqref{eqn;anal-profile} by
$
\overline{\tilde{v}}=\overline{\F [U^{-1}(t,0)v]}
$
and taking imaginary part, 
we have
\begin{align}
\label{eqn;V2}
  \frac{1}{2}
    \pt_t |\tilde{v}(t,\xi)|^2
      =&\,-\frac{|{\rm Im}\,\lam|}{|y_2(t)|^{\frac{n}{2}(p-1)}} 
            |\tilde{v}(t,\xi)|^{p+1}
    +{\rm Im}\,
      \{R \overline{\F [U^{-1}(t,0)v]}\},
\end{align}
where $R$ denotes the remainder term given by \eqref{eqn;R}
and $R$ satisfies \eqref{eqn;R-e^y}. 
Integrating \eqref{eqn;V2} for $\xi$ over $\re^n$,
we see  that for any $0<\th<\frac{n}{2} (1-p) + \gamma p$ with $\gamma = \frac12$ for $n=1$ and $\gamma =1$ for $n \geq 2$, and $t \geq T_0$,
\begin{align}
\label{eqn;OD}
   \frac{1}{2}\frac{d}{dt} 
     \|\tilde{v}(t)\|_{L^2}^2
    \le\,-\frac{|{\rm Im}\,\lam|}{|y_1(t)|^{\frac{n}{2}(p-1)}}
                       \|\tilde{v}(t)\|_{L^{p+1}}^{p+1}
               +C|y_2(t)|^{-\frac{n}{2}(p-1)}
          |Y(t)|^{-\frac\th2} \left\| u_0 \right\|^{p+1}_{H^{1,1}}.
\end{align}
Let $\al > 0 $ be sufficiently large and 
$Y_2(t)=\int_{T_0}^t |y_2(\t)|^{-\frac{n}{2}(p-1)}\,d\t$. 
From \eqref{eqn;L^2-L^4-prop},
\eqref{eqn;young} with $y = Y_2(t)$ and \eqref{eqn;OD}, 
we see that
\begin{align}
\nn  \frac{d}{dt}\{&Y_2(t)^{\al+1}
          \|\tilde{v}(t)\|_{L^2}^2\}\\
\nn      \le&\,C
        Y_2(t)^{\al}| y_2(t) |^{-\frac{n}{2}(p-1)}
           \Big\{
               \|\tilde{v}(t)\|_{L^2(|\xi| < r)}^2
               + \|\tilde{v}(t)\|_{L^2(|\xi| \ge r)}^2
           \Big\}\\
   \label{eqn;loglog1-prop}
              &\quad 
          +C Y_2(t)^{\al+1} |y_2(t) |^{-\frac{n}{2}(p-1)}
           \Big\{
                -2|\text{Im}\, \lam|
                       \|\tilde{v}(t)\|_{L^{p+1}}^{p+1}    
                +C|Y(t) |^{-\frac\th2}  \left\| u_0 \right\|^{p+1}_{H^{1,1}} 
            \Big\}\\
   \nn  %
     \le&\,CY_2(t)^{\al} | y_2(t) |^{-\frac{n}{2}(p-1)}
              \Big\{  
                 \Big(    
                       \ep Y_2(t) 
                          \| \tilde{v}(t)\|^{p+1}_{L^{p+1}}
                       +\ep^{- \frac{2}{p-1}} Y_2(t)^{- \frac{2}{p-1}}r^{n}
                  \Big)
              +  \|\tilde{v}(t)\|^2_{L^2(|\xi| \ge r)}
                \Big\}\\
      \nn   &\,+ CY_2(t)^{\al+1} |y_2(t) |^{-\frac{n}{2}(p-1)}
           \Big\{
                -2|\text{Im}\, \lam|
                       \|\tilde{v}(t)\|_{L^{p+1}}^{p+1}    
                +C|Y(t)|^{-\frac\th2} \left\| u_0 \right\|^{p+1}_{H^{1,1}}
            \Big\}
\end{align}
By taking $\ep = \ep_5 | \mathrm{Im} \lambda |$ with small $\ep_5 >0$, the first term is absorbed by the fifth term in the above estimate \eqref{eqn;loglog1-prop}, namely we have
\begin{align*}
    \frac{d}{dt}\{Y_2(t)^{\al+1}
          \|\tilde{v}(t)\|_{L^2}^2\}
  \le&\,\,Y_2(t)^{\al} | y_2(t) |^{-\frac{n}{2}(p-1)}
              \Big\{  
                 ({\ep_5} |\mathrm{Im}  \lambda | Y_2(t))^{-\frac{2}{p-1}}r^{n}
                   +  \|\tilde{v}(t)\|_{L^2(|\xi| \ge r)}^2
                \Big\}\\
         & \,\quad
          +CY_2(t)^{\alpha +1} |y_2(t)|^{-\frac{n}{2}(p-1)}
             |Y(t)|^{-\frac\th2} \left\| u_0 \right\|^{p+1}_{H^{1,1}} .
\end{align*}
We estimate the high frequency part of the solution such that 
\begin{align}\label{eqn;gensui}
 \|\tilde{v}(t)\|_{L^2(|\xi| \ge r)}^2
    \le&\,r^{-2}
            \||\xi|\tilde{v}(t)\|^2_{L^2(|\xi| \ge r)}\\
    \nn  \le&\,Cr^{-2}\|\N v(t)\|_{L^2}^2 \le Cr^{-2} \| u_0 \|_{H^{1,1}}^2 .
\end{align}
From \eqref{eqn;loglog1-prop} and \eqref{eqn;gensui}, it holds that
\begin{align}
 \nn \frac{d}{dt}\{&Y_2(t)^{\al+1}
    \|\tilde{v}(t)\|_{L^2}^2\}\\
 \label{eqn;loglog2-prop}     \le&\,CY_2(t)^{\al} | y_2(t) | ^{-\frac{n}{2}(p-1)}
              \Big\{  
                (\ep_5 |\mathrm{Im} \lambda |  Y_2(t) )^{-\frac{2}{p-1}}r^{n} 
                +  r^{-2} \| u_0\|_{H^{1,1}}^2
                \Big\}
                \\ 
  \nn    & +CY_2(t)^{\alpha +1} |y_2(t) |^{-\frac{n}{2}(p-1)}
         |Y(t)|^{-\frac\th2} \left\| u_0 \right\|^{p+1}_{H^{1,1}} .
\end{align}
By taking $r=( | \mathrm{Im} \lambda | Y_2(t) ) ^{\frac{2}{(p-1)(2+n)} } \| u_0 \|_{H^{1,1}}^{\frac{2}{n+2}} $, 
the inequality \eqref{eqn;loglog2-prop} is optimized for $r>0$
and we obtain that
\begin{align}
 \nn    \frac{d}{dt}\{Y_2(t)^{\al+1}&
     \|\tilde{v}(t)\|_{L^2}^2\}\\
 \nn                   %
      \le&\, C
            |y_2(t) |^{-\frac{n}{2}(p-1)}
            Y_2(t)^{\al-\frac{4}{(p-1)(2+n)}} | \mathrm{Im} \lambda |^{- \frac{4}{(p-1) (2+n)}}  \| u_0 \|_{H^{1,1}}^{\frac{2n}{n+2}}
      \\ 
      \label{eqn;diff-decay} &  +CY_2(t)^{\al+1}
           |y_2(t) |^{-\frac{n}{2}(p-1)} |Y(t) |^{-{\frac\th2}} 
          \left\| u_0 \right\|^{p+1}_{H^{1,1}} \\
  \nn  =&\, C
            Y'_2(t)
            Y_2(t)^{\al-\frac{4}{(p-1)(2+n)}}   | \mathrm{Im} \lambda |^{- \frac{4}{(p-1) (2+n)}}  \| u_0 \|_{H^{1,1}}^{\frac{2n}{n+2}}
      \\ 
      \nn &  +CY_2(t)^{\al+1} |y_2(t) |^{-\frac{n}{2}(p-1)} 
      |Y(t) |^{-\frac\th2}  \left\| u_0 \right\|^{p+1}_{H^{1,1}}
\end{align}
Since $y_2(t)$ satisfies \eqref{eqn;critical} or \eqref{eqn;sub-critical}, integrating \eqref{eqn;diff-decay} over $[T_0,t]$, 
we see that  
\begin{align*}
  Y_2(t)^{\al+1} \|\tilde{v}(t)\|_{L^2}^2
      \le&\, C
          Y_2(t)^{\al+1-\frac{4}{(p-1)(2+n)}}  | \mathrm{Im} 
          \lambda |^{- \frac{2}{(p-1) (2+n)}}  
          \| u_0 \|_{H^{1,1}}^{\frac{2n}{n+2}}\\
      &+C\left\| u_0 \right\|^{p+1}_{H^{1,1}}
      \int_{T_0} ^t 
          \tau^{\delta _{*}(\al+1) -(1- \delta _{*}) - \frac{\delta \th}{2} } 
       d \tau \\
    \le&\,C
          Y_2(t)^{\al+1-\frac{4}{(p-1)(2+n)}}  | \mathrm{Im} 
          \lambda |^{- \frac{2}{(p-1) (2+n)}}  
          \| u_0 \|_{H^{1,1}}^{\frac{2n}{n+2}}
    +C\left\| u_0 \right\|^{p+1}_{H^{1,1}}
           t^{\delta _{*}(\al+1) + \delta _{*}- \frac{\delta \th}{2} }\\
     \le&\,C
          Y_2(t)^{\al+1-\frac{4}{(p-1)(2+n)}}  | \mathrm{Im} 
          \lambda |^{- \frac{2}{(p-1) (2+n)}}  
          \| u_0 \|_{H^{1,1}}^{\frac{2n}{n+2}}
    +C\left\| u_0 \right\|^{p+1}_{H^{1,1}}
           Y_2(t)^{\al+1} t^{\delta _{*}- \frac{\delta \th}{2} }
\end{align*}
Hence, under the assumption $\delta > \frac{2 \del_*}{\theta}$, 
one gets that 
\begin{align*}
Y_2(t)^{\al+1}
          \|\tilde{v}(t)\|_{L^2}^2
      \le C
          Y_2(t)^{\al+1-\frac{4}{(p-1)(2+n)}}  | \mathrm{Im} \lambda |^{- \frac{2}{(p-1) (2+n)}}  \| u_0 \|_{H^{1,1}}^{\frac{2n}{n+2}}
      +C\| u_0 \|_{H^{1,1}}^2 Y_2(t)^{\alpha+1} 
        t^{\delta _{*} - \frac{\delta \th}{2}}.
\end{align*}
Therefore, we conclude by the Plancherel theorem that
\begin{align*}
\|v(t)\|_{L^2} 
        \le C 
        ( Y_2(t) |\mathrm{Im} \lambda| )^{-\frac{2}{(p-1)(2+n)}}
          \| u_0 \|_{H^{1,1}}^{\frac{2n}{n+2}}
             + C( Y_2(t)^{- \alpha -1}  
        + t^{\delta _{*
        } - \frac{\delta \theta}{2}}) \| u_0 \|_{H^{1,1}} ^2 
\end{align*}
 for any $t \geq T_0$. If $p$ is critical, noting $\delta _{*} =0$ and taking $\alpha$ enough large, we get, 
\begin{align*}
\|v(t)\|_{L^2} 
        \le C 
        ( Y_2(t) |\mathrm{Im} \lambda| )^{-\frac{2}{(p-1)(2+n)}}\| u_0 \|_{H^{1,1}}^{\frac{2n}{n+2}}.
\end{align*}
On the other hand, if $p$ is sub-critical, we have 
\begin{align*}
Y_2(t) \geq C \int_{T_0}^t \tau^{-1 + \delta _{*}} d \tau 
  \geq C t^{\delta _{*}}
\end{align*}
and hence we get 
\begin{align*}
\|v(t)\|_{L^2} 
        \le C \| u_0 \|_{H^{1,1}} ^2
        \max\{  ( t^{\delta _{*}} |\mathrm{Im} \lambda| )^{-\frac{2}{(p-1)(2+n)}} , t^{\delta _{*} - \frac{\delta \th}{2}}  \},
\end{align*}
where $0<\th<1$ is arbitrary.
\end{proof}
 
\begin{rem}
If $\sigma (t) \equiv 0$, it can be seen that, we do not need to assume the smallness for initial value $u_0$. In this case, we notice that $\delta _{*} = 1- \frac{n}{2} (p-1)$ and $\delta =1$. 
Hence we find if $\delta _{*} < \frac{\delta \theta }{2 }$, then $p >  p_{**}(n)$, 
which is smaller than $p(n) = \frac{3 + \sqrt{9+n^2 + 2n} }{n+2}$ which was found by \cite{H-L-N-CPAA-2017}. 
\end{rem}
\begin{rem}
Let $n \geq 3$ and $\sigma (t) \equiv 0$. When the decay order equilibriums, $p$ can be written as 
$p_*(n) = \frac{ n^2 + 3n + 6 + \sqrt{9n^2 + 28n + 36} }{(n+2)^2}$
and we deduce $p_{**}(n) < p_*(n) < p(n)$. Hence we see that if $p> p(n)$, the candidate of decay order is $t^{- \frac{2 \delta _{*}}{(p-1) (2+n)}}$ and which corresponds to the one in Theorem 1 in \cite{H-L-N-CPAA-2017}. 
As for the case, $p_{**}(n) < p < p_*(n)$, we have the weak decay
\begin{align*}
\left\| u(t) \right\|_{L^2} 
   \leq C 
          t^{\delta _{*} - \frac{\delta \th}{2}}.
\end{align*}
On the other hand, in $p_*(n) < p< p(n)$, we have 
\begin{align*}
\|u(t)\|_{L^2} 
        \le C 
        (t^{\delta _{*}} |\mathrm{Im} \lambda| )^{-\frac{2}{(p-1)(2+n)}} 
\end{align*}
and see that this estimate will be a natural extension of the result of \cite{H-L-N-CPAA-2017} from the region $p(n) < p < 1 + \frac{2}{n}$ to the region $p_*(n) < p < 1 + \frac{2}{n}$.
\end{rem}

\section{Models of $\sigma (t)$.}
Here, we introduce some models of $\sigma (t)$. \\ ~~ \\ 
{\bf Example 1.:} We consider the case where $\sigma (t) = \sigma _0 t^{-2} $, $t \geq T_0$, $\sigma _0 \in [0,1/4)$. Then employing 
$\mu = (1 - \sqrt{1- 4 \sigma _0})/2 \in [0,1/2)$, we have that 
\begin{align*}
y_0 (t) = \alpha t^{\mu} + \beta t^{1- \mu}.
\end{align*}
Assuming the suitable condition on $\sigma (t)$, one finds $y_1 (t) = c_1 t^{\mu}$ and $y_2(t) = c_2t^{1- \mu} + c_3 t^{\mu} $ with $c_1, c_2 \neq 0$ and $c_3 \in \R$, (see, e.g., \cite{KY-JEE-2018}). Then it follows that 
\begin{align*}
\left|
 \frac{y_1(t)}{y_2(t)}
\right| \leq C t^{1- 2 \mu}, \quad \lim_{t \to \infty} t |y_2(t)|^{-\frac{n}{2} (p-1)} \sim \lim_{t \to \infty} t |t|^{-\frac{n}{2} (p-1) (1- \mu)}
\end{align*}
i.e., $\delta = 1-2 \mu$ and $p = 1 + \frac{2}{n(1- \mu)}$ is critical. 
\\ ~~ \\ 
{\bf Example 2.:} Let $\sigma (t) = - \rho t^{- 2}$, $t \geq T_0$, $\rho \geq 0$. Then using $\theta _{-} = (1 - \sqrt{1 + 4 \rho})/2 < 0$, we have 
\begin{align*}
y_0 (t) = \alpha t^{1-\theta _-} + \beta t^{\theta _-}. 
\end{align*}
By taking suitable $\sigma (t)$, we find $y_1 (t) = c_1 t^{\theta_-}$ and $y_2(t) = c_2t^{1-\theta _-} + c_3 t^{\theta _-} $ with $c_1, c_2 \neq 0$ and $c_3 \in \R$, which tells us 
\begin{align*}
\left|
 \frac{y_1(t)}{y_2(t)}
\right| \leq C t^{1- 2 \theta _-}, \quad \lim_{t \to \infty} t |y_2(t)|^{-\frac{n}{2} (p-1)} \sim \lim_{t \to \infty} t |t|^{-\frac{n}{2} (p-1) (1- \theta _-)},
\end{align*}
i.e., $\delta = 1-2 \theta_-$ and $p = 1 + \frac{2}{n(1- \theta _-)}$ is critical. 
\\ ~~ \\ 
{\bf Example 3.:} Consider the case where $\lim_{t \to \infty} t^2\sigma (t) = 0$. In that case, it was known that there exist $c_1 \neq 0$ and $c_2 \neq 0$ such that 
\begin{align*}
\lim_{t \to \infty} y_1(t) = c_1 , \quad \lim_{t \to \infty} \frac{y_2(t)}{t} = c_2, 
\end{align*}
see, e.g., Naito \cite{Na}. In this case, we have $\delta =1$ and $p = 1 + \frac{2}{n}$ is critical.

\vskip5mm\noindent
{\bf Acknowledgment.}
The first author is supported by 
JSPS grant-in-aid for Early-Career Scientists
\#20K14328 and 
the second  author is supported by JSPS grant-in-aid for Early-Career Scientists 
\#22K13937.


\end{document}